\documentclass[a4paper,reqno,11pt]{amsart}
\usepackage{amsmath, amsbsy, amsthm, amssymb, amsfonts, latexsym, mathrsfs, color}
\usepackage{stmaryrd}
\usepackage{bm}

\usepackage{graphicx,float,anysize}         
\usepackage[top=28mm,right=28mm,bottom=28mm,left=28mm]{geometry}

\definecolor{darkblue}{rgb}{0,0,0.8}

\newtheorem{theorem}{Theorem}[section]

\newtheorem{lemma}[theorem]{Lemma}

\newtheorem{corollary}[theorem]{Corollary}

\newtheorem{Theorem}{Theorem}
\newtheorem{Corollary}[Theorem]{Corollary}

\theoremstyle{definition}

\newtheorem{remark}[theorem]{Remark}

\newtheorem*{Remark}{Remark}
\newtheorem*{Definition}{Definition}

\DeclareMathOperator{\PGL}{\mathrm{PGL}}

\DeclareMathOperator{\POmega}{\mathrm{P}\Omega}

\DeclareMathOperator{\PSL}{\mathrm{PSL}}

\newcommand{\norml}{\vartriangleleft}
\DeclareMathOperator{\Wr}{wr}

\DeclareMathOperator{\Aut}{Aut}

\DeclareMathOperator{\Sym}{Sym}

\DeclareMathOperator{\soc}{soc}

\newcommand{\la}{\langle}
\newcommand{\ra}{\rangle}

\newcommand{\normeq}{\trianglelefteqslant}

\renewcommand{\a}{\alpha}
\renewcommand{\b}{\beta}

\renewcommand{\bf}{\textbf}

 \renewcommand{\to}{\rightarrow}

\newcommand{\leqs}{\leqslant}
\newcommand{\geqs}{\geqslant}

 \newcommand{\vs}{\vspace{3mm}}

\begin{document}

\title[Permutation groups and derangements of odd prime order]{Permutation groups and derangements \\ of odd prime order}

\author{Timothy C. Burness}
\address{T.C. Burness, School of Mathematics, University of Bristol, Bristol BS8 1TW, UK}
\email{t.burness@bristol.ac.uk}

\author{Michael Giudici}\thanks{Both authors thank two anonymous referees for helpful comments. The second author is supported by the Australian Research Council Grant DP160102323}
\address{M. Giudici, School of Mathematics and Statistics, The University of Western Australia, 35 Stirling Highway, Crawley WA 6009, Australia}
\email{michael.giudici@uwa.edu.au}

\date{\today}

\begin{abstract}
Let $G$ be a transitive permutation group of degree $n$. We say that $G$ is $2'$-elusive if $n$ is divisible by an odd prime, but $G$ does not contain a derangement of odd prime order. In this paper we study the structure of quasiprimitive and biquasiprimitive $2'$-elusive permutation groups, extending earlier work of Giudici and Xu on elusive groups. As an application, we use our results to investigate automorphisms of finite arc-transitive graphs of prime valency. 
\end{abstract}

\maketitle

\section{Introduction}\label{s:intro}

Let $G \leqs {\rm Sym}(\Omega)$ be a transitive permutation group on a finite set $\Omega$ of size at least $2$. An element $x \in G$ is a \emph{derangement} if it acts fixed-point-freely on $\Omega$. Equivalently, if $H$ is a point stabiliser, then $x$ is a derangement if and only if the conjugacy class of $x$ fails to meet $H$. An easy application of the Orbit-Counting Lemma shows that $G$ contains derangements. This classical theorem of Jordan has interesting applications in number theory and topology (see Serre's article \cite{Serre}, for example).

By a theorem of Fein, Kantor and Schacher \cite{FKS}, $G$ contains a derangement of prime power order. This result turns out to have some important number-theoretic applications; for example, it implies that the relative Brauer group of any nontrivial extension of global fields is infinite (see \cite[Corollary 4]{FKS}). It is worth noting that the existence of a derangement of prime power order in \cite{FKS} requires the Classification of Finite Simple Groups. In most cases, $G$ contains a derangement of prime order, but there are some exceptions, such as the $3$-transitive action of the smallest Mathieu group ${\rm M}_{11}$ on $12$ points. The transitive permutation groups with this property are called \emph{elusive} groups, and they have been the subject of many papers in recent years; see \cite{CGJKKMN, G, GK, GMPV, GX, xu}, for example. 

A local notion of elusivity was introduced in \cite{BGW}. Let $G \leqs {\rm Sym}(\Omega)$ be a finite transitive permutation group and let $r$ be a prime divisor of $|\Omega|$. We say that $G$ is \emph{$r$-elusive} if it does not contain a derangement of order $r$ (so $G$ is elusive if and only if it is $r$-elusive for every prime divisor $r$ of $|\Omega|$). In \cite{BGW}, all the $r$-elusive primitive almost simple groups with socle an alternating or sporadic group are determined. This work has been extended in our recent book \cite{BG}, which provides a detailed study of $r$-elusive classical groups. The $r$-elusive notion leads naturally to the definition of a $2'$-elusive permutation group, which are the main focus of this paper. 

\begin{Definition}
A finite transitive permutation group $G \leqs {\rm Sym}(\Omega)$ is \emph{$2'$-elusive} if $|\Omega|$ is divisible by an odd prime, but $G$ does not contain a derangement of odd prime order. 
\end{Definition}

Let $G \leqs {\rm Sym}(\Omega)$ be a transitive permutation group with point stabiliser $H$. Recall that $G$ is \emph{primitive} if $H$ is a maximal subgroup of $G$, and note that 
every nontrivial normal subgroup of a primitive group is transitive. This observation suggests a natural generalisation of primitivity; we say that $G$ is \emph{quasiprimitive} if every nontrivial normal subgroup is transitive. Similarly, $G$ is \emph{biquasiprimitive} if every nontrivial normal subgroup has at most two orbits on $\Omega$, and there is at least one nontrivial normal subgroup with two orbits. 

Quasiprimitive and biquasiprimitive groups arise naturally in the study of finite vertex-transitive graphs. For example, if $G$ is a vertex-transitive group of automorphisms of a graph $\Gamma$ such that for each vertex $v$, the action of the vertex stabiliser $G_v$ on the set of neighbours of $v$ is quasiprimitive (that is, $\Gamma$ is \emph{$G$-locally-quasiprimitive}), then \cite[Lemma 1.6]{P85} implies that every normal subgroup $N$ of $G$ with at least three orbits is semiregular (that is, $N_v=1$ for every vertex $v$). In this situation, the quotient graph with respect to the orbits of such a normal subgroup inherits many of the symmetry properties of the original graph $\Gamma$. This explains why quasiprimitive and biquasiprimitive groups often arise as base cases in the analysis of various families of vertex-transitive graphs, see for example \cite{DGLP,P}. These important graph-theoretic applications motivated Praeger to establish detailed structure theorems for quasiprimitive \cite{P} and biquasiprimitive groups \cite{Praeger}. The structure theorem for quasiprimitive groups is similar to the celebrated O'Nan-Scott Theorem for primitive groups.

The elusive quasiprimitive permutation groups have been determined by Giudici (see \cite[Theorem 1.1]{G}); the only examples are primitive groups of the form $G={\rm M}_{11} \Wr K$ in its product action on $\Omega=\Delta^{k}$, where $K \leqs S_k$ is transitive and $|\Delta|=12$.  Further progress has been made by Giudici and Xu in \cite{GX}, where the biquasiprimitive elusive groups are determined (see \cite[Theorem 1.4]{GX}). As an application, they prove that every finite vertex-transitive, locally-quasiprimitive graph $\Gamma$ has a semiregular automorphism (in other words, the automorphism group ${\rm Aut}(\Gamma)$, viewed as a permutation group on the set of vertices of $\Gamma$, contains a derangement of prime order);  
see \cite[Theorem 1.1]{GX}. This result settles an important case of the \emph{Polycirculant Conjecture} from 1981, which asserts that every finite vertex-transitive digraph has a semiregular automorphism \cite{CGJKKMN, maru81}. For example, \cite[Theorem 1.1]{GX} immediately implies that the conjecture holds for every finite arc-transitive graph of prime valency.

The main goal of this paper is to extend this earlier work from elusive to $2'$-elusive groups. We begin by determining the primitive $2'$-elusive groups.

\begin{Theorem}\label{t:prim2dash}
Let $G \leqslant {\rm Sym}(\Omega)$ be a finite primitive permutation group.
Then $G$ is $2'$-elusive if and only if 
$$\soc(L)^k\normeq G \leqs L \Wr K$$ 
and $G$ acts with its product action on 
$\Omega = \Delta^k$ for some $k \geqs 1$, where $ L \leqs {\rm Sym}(\Delta)$ is almost simple and primitive with stabiliser $J$, $G$ induces the transitive subgroup $K \leqs S_k$ on the set of $k$ simple direct factors of $\soc(G)= \soc(L)^k$ and one of the following holds:
\begin{itemize}\addtolength{\itemsep}{0.2\baselineskip}
\item[{\rm (i)}] $L = {\rm M}_{11}$ and $J=\PSL_2(11)$; 
\item[{\rm (ii)}] $L= {}^2F_4(2)$ and $J = \PSL_2(25).2_3$.
\end{itemize}
\end{Theorem}

\begin{Remark}\label{r:prim}
Let us make some comments on the statement of Theorem \ref{t:prim2dash}.
\begin{itemize}\addtolength{\itemsep}{0.2\baselineskip}
\item[{\rm (a)}] In both cases that arise, $3$ is the only odd prime divisor of $|\Omega|$. 
\item[{\rm (b)}] In case (i), $G= {\rm M}_{11} \Wr K$, $G_{\a} = \PSL_2(11) \Wr K$  (arising from the action of ${\rm M}_{11}$ on the cosets of a subgroup $\PSL_2(11)$) and $G$ is elusive. 
\item[{\rm (c)}] In (ii), $\PSL_2(25).2_3$ is an almost simple nonsplit extension.
\item[{\rm (d)}] As noted above, the examples in (i) are the only primitive elusive groups, so Theorem \ref{t:prim2dash} shows that the $2'$-elusivity property is indeed weaker than elusivity (even for primitive groups).
\end{itemize}
\end{Remark}

Our next result describes the structure of the quasiprimitive $2'$-elusive groups (in view of Theorem \ref{t:prim2dash}, we may assume that $G$ is imprimitive). 

\begin{Theorem}\label{t:qprim2dash}
Let $G \leqslant {\rm Sym}(\Omega)$ be a finite $2'$-elusive quasiprimitive imprimitive permutation group with point stabiliser $H$. Then the following hold:
\begin{itemize}\addtolength{\itemsep}{0.2\baselineskip}
\item[{\rm (i)}] There is an almost simple group $L$ with socle $\PSL_2(p)$ for some Mersenne prime $p$, and a transitive subgroup $K\leqslant S_k$ for some positive integer $k$ such that
$$\soc(L)^k\normeq G\leqslant L \Wr K$$
and $K$ is the group induced by $G$ on the set of $k$ simple direct factors of $\soc(L)^k$. 
\item[{\rm (ii)}] Moreover, $G$ acts faithfully on a nontrivial system of imprimitivity that can be identified with $\Delta^k$, where $L$ acts transitively on $\Delta$, $\soc(L)$ has point stabiliser $C_p{:}C_{(p-1)/2}$, and
\begin{equation}\label{e:qqq}
(C_p{:}C_r)^k\leqslant (H\cap \soc(L))^k < (C_p{:}C_{(p-1)/2})^k
\end{equation}
where $r$ is the product of the distinct prime divisors of $(p-1)/2$. In particular, $(p-1)/2$ is not square-free.
\end{itemize}
\end{Theorem}

\begin{Remark}\label{r:qprim}
Some comments on the statement of Theorem \ref{t:qprim2dash}.
\begin{itemize}\addtolength{\itemsep}{0.2\baselineskip}
\item[{\rm (a)}] Notice that the final inclusion in \eqref{e:qqq} is strict because $2'$-elusivity requires $|\Omega|$ to be divisible by an odd prime. In particular, $r<(p-1)/2$ and thus $(p-1)/2$ is not square-free.
\item[{\rm (b)}] We refer the reader to Remark \ref{r:qpex}, which shows that there are genuine examples satisfying the conditions in Theorem \ref{t:qprim2dash}.
\end{itemize}
\end{Remark}

In order to state our final result, recall that a transitive group $G \leqs {\rm Sym}(\Omega)$ is biquasiprimitive if every nontrivial normal subgroup of $G$ has at most two orbits on $\Omega$, and there is a normal subgroup with two orbits, say $\Delta_1$ and $\Delta_2$. Let $G^+$ denote the index-two subgroup of $G$ that fixes $\Delta_1$ and $\Delta_2$ setwise. 

\begin{Theorem}\label{t:bqprim2dash}
Let $G \leqslant {\rm Sym}(\Omega)$ be a finite $2'$-elusive biquasiprimitive permutation group with point stabiliser $H=G_{\a}$ and minimal normal subgroup $N$. Let $K \leqs S_k$ be the transitive group induced by $G$ on the set of $k$ simple direct factors of $N$ and let $K^+ \leqs K$ be the group induced by $G^+$. 

\vspace{1mm}

\begin{itemize}\addtolength{\itemsep}{0.2\baselineskip}
\item[{\rm (a)}] Then $G^+=NH$ and $N$ is the unique minimal normal subgroup of $G$.

\vspace{1mm}

\item[{\rm (b)}] If $G^{+}$ acts faithfully on its two orbits, then one of the following holds:

\vspace{1mm}

\begin{itemize}\addtolength{\itemsep}{0.2\baselineskip}
\item[{\rm (i)}] $(G,H) = ({\rm M}_{10},A_5)$ or $({\rm Aut}(A_6),S_5)$;
\item[{\rm (ii)}]  $G={\rm M}_{11}\Wr K$, $H=\PSL_2(11)\Wr K^+$ and $|K:K^{+}|=2$;
\item[{\rm (iii)}] $N=({}^2F_4(2)')^k\normeq G\leqslant {}^2F_4(2)\Wr K$ and $H=N_{G^+}(N_{\a})$, where $N_{\a}=\PSL_2(25)^k$;
\item[{\rm (iv)}] $N=\PSL_2(p)^k\normeq G\leqslant\PGL_2(p)\Wr K$,  
$$(C_p{:}C_r)^k\leqslant N_{\alpha} <(C_p{:}C_{(p-1)/2})^k$$
and
$$H < N_{G^+}(N_\alpha) = G^+\cap ((C_p{:}C_{p-1}) \Wr K),$$
where $p$ is a Mersenne prime and $r$ is the product of the distinct prime divisors of $(p-1)/2$. 
\end{itemize}
Moreover, each group $G$ in (i), (ii) and (iii) is $2'$-elusive and biquasiprimitive.

\vspace{1mm}

\item[{\rm (c)}] If $G^{+}$ is not faithful on at least one orbit, then $k$ is even, $K^+$ is intransitive, $|K:K^+|=2$ and one of the following holds:
\begin{itemize}\addtolength{\itemsep}{0.2\baselineskip}
\item[{\rm (i)}] $G={\rm M}_{11}\Wr K$ and $H=(\PSL_2(11)^{k/2} \times {\rm M}_{11}^{k/2}){:}K^+$;
\item[{\rm (ii)}] $N = ({}^2F_4(2)')^k\normeq G\leqslant {}^2F_4(2)\Wr K$, $N_{\a} = {\rm PSL}_{2}(25)^{k/2} \times ({}^2F_4(2)')^{k/2}$ and 
$$H= N_{G^+}(N_{\a})=G^+\cap (({\rm PSL}_{2}(25).2_3)^{k/2}\times {}^2F_4(2)^{k/2}){:}K^+$$
with $G^+  = G \cap ({}^2F_4(2)\Wr K^+)$; 
\item[{\rm (iii)}] $N = \PSL_{2}(p)^k \normeq G \leqslant \PGL_2(p)\Wr K$,
$$(C_p{:}C_r)^{k/2} \times \PSL_2(p)^{k/2} \leqslant N_{\alpha} <(C_p{:}C_{(p-1)/2})^{k/2}\times \PSL_2(p)^{k/2}$$
and
$$H < N_{G^+}(N_\alpha)=G^+\cap ((C_p{:}C_{p-1})^{k/2}\times \PGL_2(p)^{k/2}){:}K^+,$$
where $G^+ = G \cap (\PGL_2(p)\Wr K^+)$ and $p$ is a Mersenne prime.
\end{itemize}
Moreover, each group $G$ in (i) and (ii) is $2'$-elusive and biquasiprimitive.
\end{itemize}
\end{Theorem}

We refer the reader to Remarks \ref{r:exd} and \ref{r:bqex} for further comments on the examples arising in parts (b)(iv) and (c)(iii) of Theorem \ref{t:bqprim2dash}, respectively.

\begin{Corollary}\label{cor:largeq}
Let $G \leqs {\rm Sym}(\Omega)$ be a finite quasiprimitive or biquasiprimitive permutation group such that $|\Omega|$ is divisible by a prime $q \geqs 5$. Then either $G$ contains a derangement of odd prime order, or ${\rm PSL}_{2}(p)^k$ is the unique minimal normal subgroup  of $G$, where $k \geqs 1$ and $p$ is a Mersenne prime such that $q^2$ divides $(p-1)/2$.
\end{Corollary}

\begin{Remark} 
Referring to Corollary \ref{cor:largeq}, it is worth noting that $2^{61}-1$ is the smallest Mersenne prime $p$ with the property that $(p-1)/2$ is divisible by $q^2$ for a prime $q \geqs 5$.
\end{Remark}

Recall that the Polycirculant Conjecture asserts that every finite vertex-transitive graph has a semiregular automorphism. The existence of such an automorphism has numerous applications. For instance, they have been used to construct Hamiltonian paths and cycles \cite{alspach}, to provide succinct representations of graphs \cite{biggs}, and to enumerate all vertex-transitive graphs of small orders \cite{mcr}. In many of these applications, it is desirable to work with a semiregular automorphism of order as large as possible. For example, the Polycirculant Conjecture is established for all vertex-transitive cubic graphs in \cite{MaruScap}, and later work of Cameron et al. \cite{CSS} shows that any such graph admits a semiregular automorphism of order greater than two. In fact, the main theorem of \cite{MSV} reveals that there is a function 
$f:\mathbb{N} \to \mathbb{N}$, satisfying $f(n)\rightarrow \infty$ as $n\rightarrow\infty$, such that any vertex-transitive cubic  graph on $n$ vertices contains a semiregular subgroup of order at least $f(n)$. 

As noted above, Giudici and Xu use their work on elusive groups in \cite{GX} to verify the Polycirculant Conjecture for all finite arc-transitive graphs of prime valency $p$.  We anticipate that our results on $2'$-elusive groups in this paper will play a key role in establishing the conjecture for arc-transitive graphs of valency $2p$. A key tool in order to achieve this goal is Theorem \ref{t:graph} below, which may be of independent interest (the application to graphs of valency $2p$ will be the subject of a future paper). This approach is similar to the aforementioned extension of the main theorem of \cite{MaruScap} in \cite{CSS}.

In the statement of the theorem, $V\Gamma$ denotes the set of vertices of $\Gamma$, and $K_{12}$ is the complete graph on $12$ vertices. In addition, the \emph{standard double cover}  of $\Gamma$ is the graph with vertex set $V\Gamma\times \{0,1\}$, such that $\{(u,a),(v,b)\}$ is an edge if and only if $a \ne b$ and $\{u,v\}$ is an edge of $\Gamma$. This graph is also known as the direct product of $\Gamma$ with $K_2$.

\begin{Theorem}\label{t:graph}
Let $\Gamma$ be a finite connected graph of prime valency $p$ and let $G\leqslant \Aut(\Gamma)$ be an arc-transitive group of automorphisms so that the action of $G$ on $V\Gamma$ is either quasiprimitive or biquasiprimitive. Then one of the following holds:
\begin{itemize}\addtolength{\itemsep}{0.2\baselineskip}
\item[{\rm (i)}] $G$ contains a derangement of odd prime order;
 \item[{\rm (ii)}] $|V\Gamma|$ is a power of $2$;
 \item[{\rm (iii)}] $\Gamma=K_{12}$, $G={\rm M}_{11}$ and $p=11$;
 \item[{\rm (iv)}] $|V\Gamma|=(p^2-1)/2s$ and $G=\PSL_2(p)$ or $\PGL_2(p)$, where $p$ is a Mersenne prime and $C_r \leqs C_s < C_{(p-1)/2}$, where $r$ is the product of the distinct prime divisors of $(p-1)/2$;
 \item[{\rm (v)}] $|V\Gamma|=(p^2-1)/s$ and $G=\PGL_2(p)$, where $p$ and $s$ are as in part (iv), and $\Gamma$ is the standard double cover of the graph given in (iv).
\end{itemize}
\end{Theorem} 

As we will explain in Section \ref{s:graphs}, if $G \leqs {\rm Sym}(\Omega)$ is a finite transitive permutation group then there is a one-to-one correspondence between the set of suborbits of $G$ and the set of finite digraphs with vertex set $\Omega$ on which $G$ acts arc-transitively. Moreover, the connected graphs of valency $p$ correspond to self-paired suborbits $\omega^{G_{\a}}$ of length $p$ with the property that $G = \la G_{\a},g\ra$ for each $g \in G$ that interchanges $\a$ and $\omega$. Therefore, one of the main steps in the proof of Theorem \ref{t:graph} is to determine the $2'$-elusive quasiprimitive and biquasiprimitive groups with a prime subdegree; we can do this by applying Theorems \ref{t:prim2dash},  \ref{t:qprim2dash} and \ref{t:bqprim2dash}. In the cases that arise, we then need to check that $G$ has a suborbit with the appropriate properties.

Finally, we record a couple of corollaries to Theorem \ref{t:graph} (the short proofs are presented at the end of Section \ref{s:graphs}). 

\begin{Corollary}\label{c:prime}
Let $\Gamma$ be a finite connected graph of prime valency and let $G \leqslant \Aut(\Gamma)$ be an elusive arc-transitive group of automorphisms. Then $\Gamma=K_{12}$ and $G={\rm M}_{11}$.
\end{Corollary}

\begin{Corollary}\label{c:6}
The smallest integer $k$ such that there is a finite connected graph of valency $k$ with an elusive arc-transitive group of automorphisms is $6$.
\end{Corollary}

Note that Corollary \ref{c:6} answers a question posed in \cite{GMPV}. The smallest $k$ for which there is a finite connected graph of valency $k$ with an elusive \emph{vertex}-transitive group of automorphisms is still unknown.

\vs

\noindent \textbf{Notation.} Our notation is standard. We write $H.K$ to denote an extension of $H$ by $K$, and $H{:}K$ if the extension splits. If $n$ is a positive integer then $C_n$ denotes a cyclic group of order $n$, and $H^n$ is the direct product of $n$ copies of $H$. If $p$ is a prime, then ${\rm O}_{p}(H)$ denotes the largest normal $p$-subgroup of $H$. Finally, if $H$ acts on a set $\Delta$ then we write $H^{\Delta}$ to denote the induced permutation group on $\Delta$.

\section{Simple groups}\label{s:simple}
 
In \cite[Theorem 1.3]{G}, Giudici determines the nonabelian finite simple groups $T$ with a proper subgroup that meets every 
$\Aut(T)$-conjugacy class of elements of prime order. We can adopt a similar approach in order to establish an analogous result for odd primes, which will play a key role in the proofs of our main theorems.  

\begin{remark}\label{r:thm}
In the first row of Table \ref{tab:alloddautT}, $p$ is a Mersenne prime and $r$ is the product of the distinct prime divisors of $(p-1)/2$. Also observe that $|T:H|$ is a $2$-power if $H = C_p{:}C_{(p-1)/2}$, so in this case the action of $T$ on the cosets of $H$ is not $2'$-elusive (recall that  for $2'$-elusivity, the degree must be divisible by an odd prime). 
\end{remark}

\begin{theorem}\label{thm:autT}
Let $T$ be a nonabelian finite simple group.
\begin{itemize}\addtolength{\itemsep}{0.2\baselineskip}
\item[{\rm (i)}] $T$ has a proper subgroup $H$ that meets every $\Aut(T)$-class of elements of odd prime order in $T$ if and only if $(T,H)$ is one of the cases in Table \ref{tab:alloddautT}. 
\item[{\rm (ii)}] In addition, $H$ meets every $T$-class of elements of odd prime order in $T$ if and only if $T={}^2F_4(2)'$, ${\rm M}_{11}$ or $\PSL_2(p)$ with $p$ a Mersenne prime.
\end{itemize}
\end{theorem}

\renewcommand{\arraystretch}{1.1}
\begin{table}
$$\begin{array}{ll} \hline
T & H  \\ \hline
\PSL_2(p) & C_p{:}C_r\leqslant H\leqslant C_p{:}C_{(p-1)/2} \;\; \mbox{(see Remark \ref{r:thm})} \\
\POmega^+_{8}(3) & \Omega_7(3) \\
\Omega^+_8(2) & {\rm Sp}_{6}(2)  \\
\Omega^+_8(2) & A_9 \\
{}^2F_4(2)' & \PSL_2(25)  \\
A_6 & A_5 \\
{\rm M}_{11} & \PSL_2(11)  \\ \hline
\end{array}$$
\caption{The cases $(T,H)$ in Theorem \ref{thm:autT}(i)}
\label{tab:alloddautT}
\end{table}
\renewcommand{\arraystretch}{1}

\begin{proof}
Suppose $H<T$ is a proper subgroup that meets every $\Aut(T)$-class of elements  of odd prime order in $T$, so every odd prime divisor of $|T|$ also divides $|H|$. Moreover, if $H\leqslant K\leqslant T$ then every $\Aut(T)$-class of elements of odd prime order in $T$ meets $K$. Thus we will initially assume that $H$ is a maximal subgroup of $T$; if $(T,H)$ is an example then we will need to check if any proper subgroups of $H$ also meet every $\Aut(T)$-class of elements of odd prime order.

First assume that $T$ is a sporadic simple group. Here the possibilities for $T$ and $H$ (with $H$ maximal and $|H|$ divisible by every odd prime divisor of $|T|$) can be read off from \cite[Table 10.6]{LPS}:
$$\begin{array}{llll}
({\rm M}_{11},\PSL_{2}(11)) & ({\rm M}_{12},\PSL_{2}(11)) & ({\rm M}_{12},{\rm M}_{11}) & ({\rm M}_{24},{\rm M}_{23}) \\
({\rm HS}, {\rm M}_{22}) & ({\rm McL}, {\rm M}_{22}) & ({\rm Co}_{2}, {\rm M}_{23}) & ({\rm Co}_{3}, {\rm M}_{23}) 
\end{array}$$
It follows that $\pi(H)=\pi(T)$, where $\pi(X)$ is the set of prime divisors of $|X|$. The cases
$({\rm McL}, {\rm M}_{22})$ and $({\rm Co}_{2}, {\rm M}_{23})$ are ruled out in \cite[Section 3.11]{G}, where an $\Aut(T)$-class of elements of odd prime order not meeting $H$ is identified. If $T= {\rm M}_{11}$, ${\rm M}_{24}$ or ${\rm Co}_3$ then $T = {\rm Aut}(T)$ and by applying \cite[Corollary 1.2]{BGW} we deduce that $({\rm M}_{11}, \PSL_2(11))$ is the only example. In addition, no proper subgroup of $\PSL_2(11)$ has order divisible by every  odd prime divisor of $|{\rm M}_{11}|$, so no further examples arise. In the remaining three cases, we can use \cite{atlas} to identify an $\Aut(T)$-class of elements of odd prime order $r$ that does not meet $H$ (indeed, take $r=5$ if $(T,H)=({\rm HS}, {\rm M}_{22})$, and $r=3$ in the other two cases).

Next assume $T=A_n$ is an alternating group. Since the two largest primes at most $n$ must divide $|H|$, \cite[Theorem 4]{LPS} implies that $H\cong (S_k\times S_{n-k})\cap T$ for some $1 \leqs k<n/2$ (note that this includes the case $(T,H) = (A_6,\PSL_2(5))$). In particular, the action of $T$ on the set of right cosets of $H$ is permutation isomorphic to the action of $T$ on the set of subsets of $\{1,\ldots,n\}$ of size $k$. Since this action extends to $S_n$, it follows that if $n \neq 6$ then the $\Aut(T)$-class of an element $t \in T$ meets $H$ if and only if $t$ fixes a $k$-set. By a theorem of Sylvester \cite{sylvester}, $\binom{n}{k}$  is divisible by an odd prime, so \cite[Corollary 3.2(iii)]{BGW} implies that there is an $\Aut(T)$-class of elements of odd prime order that does not meet $H$. Finally, if $n=6$ then it is easy to check that the only subgroups $H$ of $T$ with the required property are isomorphic to $A_5$. In addition, note that $A_5$ has a unique class of elements of order $3$, but $A_6$ has two, so $H$ does not meet every $T$-class of elements of odd prime order.

For the remainder, we may assume that $T$ is a simple group of Lie type. By \cite[Theorem 4(i)]{LPS}, the possibilities for $T$ and $H$ can be read off from \cite[Tables 10.1--10.5]{LPS}. More precisely, these tables give the proper subgroups $M$ of $T$ with the property that $|M|$ is divisible by a specific collection of odd prime divisors of $|T|$. 
By inspection, and recalling that we are assuming $H$ is maximal, we deduce that either $|H|$ is even, or $(T,H) = (\PSL_3(3),C_{13}{:}C_3)$ or $(\PSL_2(p),C_p{:}C_{(p-1)/2})$ for a Mersenne prime $p$. 

If $(T,H) = (\PSL_3(3),C_{13}{:}C_3)$ then $T$ has two $\Aut(T)$-classes of subgroups of order $3$, and $H$ has a unique such class, so there is an $\Aut(T)$-class of elements of order $3$ that misses $H$. Now assume $(T,H) = (\PSL_2(p),C_p{:}C_{(p-1)/2})$ with $p$ a Mersenne prime. Here $T$ has a unique class of subgroups of each odd prime order, hence every $T$-class of elements of odd prime order meets $H$. The same conclusion holds for any subgroup $L<H$ with $\pi(H)=\pi(L)$, so we deduce that  
$$C_p{:}C_r\leqslant H\leqslant C_p{:}C_{(p-1)/2}$$
as in the first row of Table \ref{tab:alloddautT} (where $r$ is the product of the distinct prime divisors of $(p-1)/2$).

To complete the proof we may assume that $|H|$ is even and thus $\pi(T)=\pi(H)$. Here the possibilities for $(T,H)$ can be read off from \cite[Table 10.7]{LPS}. These cases were studied in \cite[Section 3]{G}, where in most instances an $\Aut(T)$-class of elements of odd prime order that misses $H$ is identified. The exceptions are as follows:
$$\begin{array}{lll}
({\rm P\Omega}_{8}^{+}(3), \Omega_7(3)) & (\Omega_{8}^{+}(2), A_9) & (\Omega_{8}^{+}(2), {\rm Sp}_{6}(2)) \\
({\rm PSU}_{6}(2), {\rm M}_{22}) & ({\rm PSU}_{5}(2), {\rm PSL}_{2}(11)) &  ({}^2F_4(2)', {\rm PSL}_{2}(25)) 
\end{array}$$
For the cases with $T = {\rm PSU}_{6}(2)$ or ${\rm PSU}_{5}(2)$, one can use {\sc Magma} \cite{magma}, or the information in \cite{atlas}, to check that there is an $\Aut(T)$-class of elements of order $3$ that misses $H$. The remaining four cases are recorded in Table \ref{tab:alloddautT}. In each of these cases it is easy to see that $H$ does not contain a proper subgroup with the desired property, so we do not obtain any additional examples. 
Finally, as explained in \cite[Sections 3.8--3.9]{G}, if $(T,H)$ is one of the examples in Table 
\ref{tab:alloddautT} with $T = \Omega^+_8(2)$ or ${\rm P\Omega}^+_8(3)$ then $H$ does not meet every $T$-class of elements of odd prime order, so $(T,H)$ does not arise in part (ii) of Theorem \ref{thm:autT}.
\end{proof}

By applying Theorem \ref{thm:autT}, we can determine all the $2'$-elusive almost simple groups. In Table \ref{tab:2dashelusive}, as before, $p$ is a Mersenne prime and $r$ is the product of the distinct prime divisors of $(p-1)/2$.

\begin{theorem}\label{t:2dashas}
Let $G \leqs {\rm Sym}(\Omega)$ be a finite transitive almost simple permutation group with point stabiliser $H$. Then $G$ is $2'$-elusive if and only if $(G,H)$ is one of the cases in Table \ref{tab:2dashelusive}.
\end{theorem}

\renewcommand{\arraystretch}{1.1}
\begin{table}
$$\begin{array}{ll} \hline
G & H  \\ \hline
\mbox{$\PSL_2(p)$ or $\PGL_2(p)$} & C_p{:}C_r\leqslant H < C_p{:}C_{(p-1)/2} \\
\PGL_2(p) & C_p{:}C_{2r} \leqslant H < C_p{:}C_{p-1} \\
\mbox{${}^2F_4(2)'$ or ${}^2F_4(2)$} & \PSL_2(25)  \\
{}^2F_4(2) & \PSL_2(25).2_3 \\
\mbox{${\rm M}_{10}$ or $\Aut(A_6)$} & A_5  \\
\Aut(A_6) & S_5 \\
{\rm M}_{11} & \PSL_2(11)  \\ \hline
\end{array}$$
\caption{The $2'$-elusive almost simple groups}
\label{tab:2dashelusive}
\end{table}
\renewcommand{\arraystretch}{1}

\begin{proof}
Let $T$ denote the socle of $G$ and assume that $G$ is $2'$-elusive. Then $H\cap T$ meets every $G$-class of elements in $T$ of odd prime order, so $(T,H \cap T)$ is one of the cases arising in Theorem \ref{thm:autT}(i).

First assume $T$ is a transitive subgroup of $G$. Here $G=TH$ and thus $T$ is $2'$-elusive since $H$ meets every $T$-class of elements of odd prime order in $T$. Note that $|G:T| = |H:H \cap T|$. By applying Theorem \ref{thm:autT}(ii) we deduce that $(G,H)$ is one of the following:
\begin{itemize}\addtolength{\itemsep}{0.2\baselineskip}
\item[{\rm (a)}] $G={\rm M}_{11}$, $H={\rm PSL}_{2}(11)$
\item[{\rm (b)}] $G={}^2F_4(2)'$, $H={\rm PSL}_{2}(25)$
\item[{\rm (c)}] $G={}^2F_4(2)$, $H={\rm PSL}_{2}(25).2_3$
\item[{\rm (d)}] $G={\rm PSL}_{2}(p)$, $C_p{:}C_r \leqs H < C_p{:}C_{(p-1)/2}$
\item[{\rm (e)}] $G={\rm PGL}_{2}(p)$, $C_p{:}C_{2r} \leqs H < C_p{:}C_{p-1}$
\end{itemize}
where $p$ is a Mersenne prime and $r$ is the product of the distinct prime divisors of $(p-1)/2$. Note that in case (d) (and similarly in (e)) we require $H < C_p{:}C_{(p-1)/2}$ since $|\Omega|$ is divisible by an odd prime. Also observe that $G$ is primitive in cases (a), (b) and (c), and quasiprimitive (and imprimitive) in cases (d) and (e).

Now assume $T$ is intransitive, in which case the orbits of $T$ on $\Omega$ have equal size and the actions of $T$ on each orbit are isomorphic. Clearly, $T \ne {\rm M}_{11}$. If $T={}^2F_4(2)'$ or ${\rm PSL}_{2}(p)$ (with $p$ a Mersenne prime) then $(G,H)$ is one of the following:
\begin{itemize}\addtolength{\itemsep}{0.2\baselineskip}
\item[{\rm (f)}] $G={}^2F_4(2)$, $H={\rm PSL}_{2}(25)$
\item[{\rm (g)}] $G={\rm PGL}_{2}(p)$, $C_p{:}C_{r} \leqs H < C_p{:}C_{(p-1)/2}$
\end{itemize}
where $r$ is the product of the distinct prime divisors of $(p-1)/2$ as before. Next suppose 
$T=\Omega^+_8(2)$ or $\POmega^+_8(3)$. As explained in \cite[Section 4]{G}, $G$ must contain a triality graph automorphism 
(if not, there are derangements of order $5$), but this implies that $G$ contains an element of order $3$ that permutes the orbits of $T$, which is a derangement. Finally, let us assume $T=A_6$, so $G \in \{S_6, {\rm M}_{10}, {\rm PGL}_{2}(9), {\rm Aut}(A_6)\}$. Here $G$ is $2'$-elusive if and only if the two $T$-classes of elements of order $3$ are fused in $G$, so $(G,H)$ is one of the following:
\begin{itemize}\addtolength{\itemsep}{0.2\baselineskip}
\item[{\rm (h)}] $G={\rm M}_{10}$, $H =A_5$
\item[{\rm (i)}] $G={\rm Aut}(A_6)$, $H \in \{A_5, S_5\}$
\end{itemize}
This completes the proof of Theorem \ref{t:2dashas}.
\end{proof}

It is worth recording the cases in Theorem \ref{t:2dashas} that arise when $G$ is primitive.

\begin{corollary}\label{c:2dashas}
Let $G \leqs {\rm Sym}(\Omega)$ be a finite primitive almost simple permutation group with point stabiliser $H$. Then $G$ is $2'$-elusive if and only if 
$$\mbox{$(G,H) = ({\rm M}_{11},{\rm PSL}_{2}(11))$, $({}^2F_4(2)',{\rm PSL}_{2}(25))$ or $({}^2F_4(2),{\rm PSL}_{2}(25).2_3)$.}$$
\end{corollary}

\section{Quasiprimitive groups}\label{s:quasi}

In this section we investigate the structure of $2'$-elusive quasiprimitive groups. Our aim is to  prove Theorems \ref{t:prim2dash} and \ref{t:qprim2dash}. We begin by recording a lemma which will be useful later.  

\begin{lemma}\label{lemPA}
Let $G \leqs {\rm Sym}(\Omega)$ be a finite permutation group with a transitive normal subgroup $N=T^k$ such that $C_G(N)=1$, where $T$ is a  nonabelian simple group and $k \geqs 2$. Let $\alpha\in \Omega$ and assume that $N_\alpha=S^k$ for some proper subgroup $S<T$. Then we can identify $\Omega$ with the Cartesian product $\Delta^k$, where $\Delta=T/S$, such that $G$ is permutation isomorphic to a subgroup of $\Aut(T)\Wr S_k$  acting on $\Delta^k$ with its usual product action.
\end{lemma}

\begin{proof}
First observe that $G$ is isomorphic to a subgroup $L$ of $\Aut(N) = \Aut(T)\Wr S_k$ since $C_G(N)=1$.  
Since $N$ acts transitively on $\Omega$ we have $G=NG_\alpha$. Let $\Sigma$ be the set of 
right cosets of $N_\alpha$ in $N$.

Define an action of $G$ on $\Sigma$ by $(N_\alpha m)^{ng}=N_\alpha (mn)^g$ for each $m,n\in N$ and $g\in G_\alpha$.  To see that this is well-defined, first observe that if 
$N_\alpha m_1=N_\alpha m_2$ then $m_1m_2^{-1}\in N_\alpha$. Therefore, since $N_\alpha\norml G_\alpha$, we deduce that
$$(m_1n)^g((m_2n)^g)^{-1}= m_1^g(m_2^{-1})^g=(m_1m_2^{-1})^g\in N_\alpha$$ 
and thus $(N_\alpha m_1)^{ng}=(N_\alpha m_2)^{ng}$. In addition, if $n_1g_1=n_2g_2$ then  $n_2^{-1}n_1=g_2g_1^{-1}\in N_\alpha$ and it follows that $g_2^{-1}g_1\in N_\alpha$ since $N_\alpha\norml G_\alpha$.  A routine calculation now shows that $(N_\alpha m)^{n_1g_1}=(N_\alpha m) ^{n_2g_2}$ and so the action of $G$ on $\Sigma$ is well-defined. Since the stabiliser in $G$ of the trivial coset $N_\alpha \in \Sigma$ is $G_\alpha$, it follows that the action of $G$ on $\Omega $ is permutation isomorphic to the action of $G$ on $\Sigma$. 

Let $\varphi:G \to L$ be the isomorphism induced by the conjugation action of $G$ on $N$. 
Let $\Lambda$ be the set of right cosets of $S^*$ in $T^*$, where $T^* = {\rm Inn}(T)$ and $S^*$ is the group of automorphisms of $T$ induced by conjugation by elements of  $S$. Let $\rho:\Sigma \to \Lambda^k$ be the bijection sending $N_{\a}(t_1, \ldots, t_k)$ to $(S^*t^*_1, \ldots, S^*t^*_k)$, where $t_i^*$ is the inner automorphism of $T$ induced by conjugation by the element $t_i \in T$. Now $L$ acts on $\Lambda^k$ via the product action: if $x=ng \in G$, with $n \in N$ and $g \in G_{\a}$, then $\varphi(n) = (a_1, \ldots, a_k) \in {\rm Inn}(T)^k$ with 
$$(S^*t^*_1,\ldots,S^*t^*_k)^{\varphi(n)}=(S^*t^*_1a_1,\ldots,S^*t^*_ka_k)$$
and 
 $\varphi(g) = (b_1, \ldots, b_k)\pi \in {\rm Aut}(T) \Wr S_k$ with
$$(S^*t^*_1, \ldots, S^*t^*_k)^{\varphi(g)} = (S^*(t^*_{1^{\pi^{-1}}})^{b_{1^{\pi^{-1}}}}, \ldots, S^*(t^*_{k^{\pi^{-1}}})^{b_{k^{\pi^{-1}}}}).$$
One checks that $\rho(\omega^x) = \rho(\omega)^{\varphi(x)}$ for all $\omega \in \Sigma$ and all $x \in G$, hence the actions of $G$ and $L$ on $\Omega$ and $\Lambda^k$, respectively, are permutation isomorphic. Finally, by identifying $\Lambda^k$ with $\Delta^k$, where $\Delta$ is the set of right cosets of $S$ in $T$, we deduce that the permutation groups $G \leqs {\rm Sym}(\Omega)$ and $L \leqs {\rm Sym}(\Delta^k)$ are permutation isomorphic.
\end{proof}

We also need the following easy lemma (the proof of \cite[Theorem 4.1(e)]{CGJKKMN} goes through unchanged).

\begin{lemma}\label{lem:wr}
Let $L \leqs {\rm Sym}(\Delta)$ be a finite $2'$-elusive permutation group and let $K \leqs S_k$ be a transitive subgroup, where $k \geqs 2$. Then the product action of $L \Wr K$ on $\Delta^k$ is also $2'$-elusive.
\end{lemma}

\begin{lemma}\label{l:socle}
Let $N =T^k \leqs {\rm Sym}(\Omega)$ be a finite transitive permutation group with point stabiliser $H$, where $T$ is simple and $k \geqs 1$. Then $N$ is $2'$-elusive if and only if one of the following holds:
\begin{itemize}\addtolength{\itemsep}{0.2\baselineskip}
\item[{\rm (i)}] $T={\rm M}_{11}$ and $H=\PSL_2(11)^k$;
\item[{\rm (ii)}] $T={}^2F_4(2)'$ and   $H=\PSL_2(25)^k$;
\item[{\rm (iii)}] $T=\PSL_2(p)$ and $(C_p{:}C_r)^k\leqslant H <(C_p{:}C_{(p-1)/2})^k$, where $p$ is a Mersenne prime and $r$ is the product of the distinct prime divisors of 
$(p-1)/2$.
\end{itemize}
\end{lemma}

\begin{proof}
By applying Theorem \ref{thm:autT}(ii), we deduce that $N$ is $2'$-elusive if (i), (ii) or (iii) holds. For the remainder, let us assume $N$ is $2'$-elusive. First observe that $H$ meets every $N$-class of elements of odd prime order. Since $H$ is core-free in $N$, it follows that $T$ is nonabelian (indeed, if $T$ is abelian then $H=1$ and $T=C_2$, which is incompatible with the fact that $|\Omega|$ is divisible by an odd prime). 

Write $N=T_1 \times \cdots \times T_k$ and $H_i = H \cap T_i$. Let $\pi_i:N \to T_i$ be the $i$-th projection map. If $C$ is a conjugacy class of $T$ then the corresponding subset of $T_i$ is a conjugacy class of $N$, so $H_i$ meets every $T_i$-class of elements of odd prime order. Since $H$ is core-free in $N$, $H_i$ is a proper subgroup of $T_i$ and thus $(T_i,H_i)$ is one of the cases arising in Theorem \ref{thm:autT}(ii). In particular, $T_i={\rm M}_{11}$, ${}^2F_4(2)'$ or ${\rm PSL}_{2}(p)$ with $p$ a Mersenne prime.

For each $i$ we have $H_i \normeq \pi_i(H) < T_i$ (note that $\pi_i(H)<T_i$ since $T_i$ is simple). If $T_i={\rm M}_{11}$ then $H_i = {\rm PSL}_{2}(11)$ is a maximal subgroup of $T_i$, so in this case $H_i = \pi_i(H) = {\rm PSL}_{2}(11)$ and thus $H={\rm PSL}_{2}(11)^k$ as in part (i) of the lemma. By the same argument, we deduce that $H = 
{\rm PSL}_{2}(25)^k$ if $T_i={}^2F_4(2)'$.  

Finally, let us assume $T_i = {\rm PSL}_{2}(p)$, where $p$ is a Mersenne prime. Here 
$$C_p{:}C_r \leqs H_i \leqs C_p{:}C_{(p-1)/2}$$ 
where $r$ is the product of the distinct prime divisors of $(p-1)/2$. Since any overgroup of $C_p{:}C_r$ in $\PSL_2(p)$ is contained in $C_p{:}C_{(p-1)/2}$, it follows that
$$C_p{:}C_r \leqs H_i\leqslant \pi_i(H) \leqs C_p{:}C_{(p-1)/2}.$$
Therefore $H$ is as given in part (iii), and we note that $H< (C_p{:}C_{(p-1)/2})^k$ since $|\Omega|$ is divisible by an odd prime.
\end{proof}

We are now in a position to prove Theorems \ref{t:prim2dash} and \ref{t:qprim2dash}.

\begin{proof}[Proof of Theorems \ref{t:prim2dash} and \ref{t:qprim2dash}]
Let $G \leqs {\rm Sym}(\Omega)$ be a finite $2'$-elusive quasiprimitive permutation group with socle $N$ and point stabiliser $H=G_{\a}$.  We claim that $N$ is the unique minimal normal subgroup of $G$. To see this, suppose that $N_1$  and $N_2$ are distinct minimal normal subgroups of $G$. Then $N_1$ and $N_2$ commute, so they are regular and nonabelian by \cite[Theorem 4.2A]{DM}. In particular, $N_1\cong T^k$ for some nonabelian simple group $T$ and positive integer $k$, so $N$ contains derangements of odd prime order, but this is incompatible with the fact that $G$ is $2'$-elusive. Therefore, $N$ is the unique minimal normal subgroup of $G$.

Write $N=T_1\times\cdots\times T_k$ for some positive integer $k$ such that $T_i\cong T$ for some simple group $T$. Since $G$ is quasiprimitive, it follows that $N$ is transitive and thus $2'$-elusive, so Lemma  \ref{l:socle} implies that one of the following holds (in particular, $N$ is nonabelian):
\begin{itemize}\addtolength{\itemsep}{0.2\baselineskip}
\item[{\rm (a)}] $T={\rm M}_{11}$ and $N_{\a}=\PSL_2(11)^k$;
\item[{\rm (b)}] $T={}^2F_4(2)'$ and $N_{\a}=\PSL_2(25)^k$;
\item[{\rm (c)}] $T=\PSL_2(p)$ and $(C_p{:}C_r)^k\leqslant N_{\a} <(C_p{:}C_{(p-1)/2})^k$, where $p$ is a Mersenne prime and $r$ is the product of the distinct prime divisors of 
$(p-1)/2$.
\end{itemize}

Since $N$ is the unique minimal normal subgroup of $G$, it follows that $C_G(N)=1$ and thus $G\leqslant \Aut(N)=\Aut(T)\Wr S_k$. Let $K \leqs S_k$ be the group induced by $G$ on the set of $k$ simple direct factors of $N$. Then 
$$T^k \leqs G \leqslant \Aut(T)\Wr K$$ 
and the minimality of $N$ implies that $K$ is transitive. 
Also note that $G=NH$, so $H$ also induces the group $K$ on the set of $k$ simple direct factors of $N$.

If (a) holds then $G = {\rm M}_{11} \Wr K$ is the only possibility (since $\Aut({\rm M}_{11})={\rm M}_{11}$), so $H = {\rm PSL}_{2}(11) \Wr K$. In view of Lemma \ref{lemPA}, we may identify $\Omega$ with $\Delta^k$, where $\Delta$ is the set of right cosets of ${\rm PSL}_{2}(11)$ in ${\rm M}_{11}$, so $G$ is a primitive product-type group as in Theorem \ref{t:prim2dash}(i). In addition, we note that any group of this form is primitive and elusive (and therefore $2'$-elusive since $|\Omega|$ is divisible by $3$).

Next assume (b) holds. Set $X=N_{\Sym(\Omega)}(N)$ and observe that $X={}^2F_4(2)\Wr S_k$ and $X_\alpha=(\PSL_2(25).2_3) \Wr S_k$. Since $G$ induces the transitive subgroup $K \leqs S_k$, it follows that $G\leqslant {}^2F_4(2)\Wr K \leqslant X$ and $H=G \cap X_\alpha$. By applying Lemma \ref{lemPA} we can identify $\Omega$ with $\Delta^k$, where $\Delta$ is the set of right cosets of $\PSL_2(25)$ in ${}^2F_4(2)'$, and we see that $G$ is a primitive product-type group as in Theorem \ref{t:prim2dash}(ii). By combining Theorem \ref{t:2dashas} and Lemma \ref{lem:wr}, we deduce that any primitive group of this form is indeed $2'$-elusive.

Finally, suppose that (c) holds. Let $\pi_i:N \to T_i$ be the $i$-th projection map.  For each $i\in\{1,\ldots,k\}$, set $R_i = \pi_i(N_\alpha)$, so 
$$C_p{:}C_r\leqslant R_i \leqslant C_p{:}C_{(p-1)/2}< T_i.$$
Since $H$ normalises $N_\alpha$ and acts transitively on the set of $k$ simple direct factors of $N$, it follows that $R_i\cong R_j$ for all $i,j$. Moreover, $H$ normalises the subgroup $R=R_1\times\cdots\times R_k$ of $N$. For each $i$, let $J_i=N_{T_i}({\rm O}_p(R_i)) = C_p{:}C_{(p-1)/2}$ and note that $H$ normalises the subgroup $J=J_1\times\cdots\times J_k$ of $N$. Moreover, $N_\alpha \leqslant R \leqs J \leqs N$ and $N_\alpha$ is a subdirect product of $R$. Also note that 
$N_\alpha\neq J$ since $|\Omega|$ is divisible by an odd prime. Therefore, $H< JH< G$ and thus $G$ preserves a nontrivial system of imprimitivity $\mathcal{P}$ of $\Omega$ such that the stabiliser of the block containing $\alpha$ is $JH$. Note that $JH \cap N=J$. The kernel of the action of $G$ on $\mathcal{P}$ is an intransitive normal subgroup of $G$, so this action is faithful by the quasiprimitivity of $G$. Finally, by applying Lemma \ref{lemPA} we can identify $\mathcal{P}$ with the Cartesian product $\Delta^k$, where $\Delta$ is the set of right cosets of $C_p{:}C_{(p-1)/2}$ in ${\rm PSL}_{2}(p)$.

\vs

This completes the proof of Theorems \ref{t:prim2dash} and \ref{t:qprim2dash}.
\end{proof}

\begin{remark}\label{r:qpex}
Let $G=\PSL_2(p)\Wr S_k$ and $H=(C_p{:}C_r) \Wr S_k$, where $p$ is a Mersenne prime and $r$ is the product of the distinct prime divisors of $(p-1)/2$. In addition, let us assume that $r \ne (p-1)/2$ (note that $p=2^7-1$ is the smallest Mersenne prime with this property). Then the action of $G$ on the set $\Omega = G/H$ of right cosets of $H$ is quasiprimitive. Moreover, Lemma \ref{lemPA} implies that the action of $G$ on $\Omega$ can be identified with the usual product action of $G$ on $\Delta^k$, where $\Delta$ is the set of right cosets of $C_p{:}C_r$ in $\PSL_2(p)$. Then by applying Theorem \ref{t:2dashas} and Lemma \ref{lem:wr}, we deduce that the action of $G$ on $\Omega$ is $2'$-elusive. This shows that the set-up described in Theorem \ref{t:qprim2dash} does give rise to genuine examples.
\end{remark}

\section{Biquasiprimitive groups}\label{s:bquasi}

In this section we turn our attention to biquasiprimitive permutation groups; our aim is to prove Theorem \ref{t:bqprim2dash}. Recall that a transitive permutation group $G \leqs {\rm Sym}(\Omega)$ is \emph{biquasiprimitive} if every nontrivial normal subgroup of $G$ has at most two orbits and there is some normal subgroup with two orbits $\Delta_1$ and $\Delta_2$. Fix such a normal subgroup and let $G^+$ denote the index-two subgroup of $G$ that fixes $\Delta_1$ and $\Delta_2$ setwise, so $\Omega = \Delta_1 \cup \Delta_2$ is a $G$-invariant partition of $\Omega$.

Recall from the introduction that the elusive biquasiprimitive groups have been determined by Giudici and Xu (see \cite[Theorem 1.4]{GX}).  Our goal is to extend this result to $2'$-elusive groups.

\begin{lemma}\label{lem:inGplus}
Let $G \leqs {\rm Sym}(\Omega)$ be a finite $2'$-elusive biquasiprimitive permutation group with point stabiliser $H$ and let $N$ be a minimal normal subgroup of $G$. Then $G^+ = NH$.
\end{lemma}

\begin{proof}
If $N \leqslant G^+$ then the biquasiprimitivity of $G$ implies that $N$ acts transitively on each $G^+$-orbit and thus $G^+ = NH$. Seeking a contradiction, suppose that $N\not\leqslant G^+$. Then by the minimality of $N$ we have $N\cap G^+=1$. Since $|G:G^+|=2$ it follows that $|N|=2$ and $G=G^+\times C_2$. Each orbit of $N$ has size $2$ and thus $|\Omega| \in \{2,4\}$ since $G$ is biquasiprimitive. But this contradicts the fact that $|\Omega|$ is divisible by an odd prime (because $G$ is $2'$-elusive). The result follows. 
\end{proof}

We now partition the proof of Theorem \ref{t:bqprim2dash} into two parts, according to whether or not $G^{+}$ acts faithfully on its orbits $\Delta_1$ and $\Delta_2$.

\subsection{$G^{+}$ acts faithfully on both orbits}\label{ss:faith}

\begin{lemma}\label{lem:NGplusfaith}
Let $G\leqs {\rm Sym}(\Omega)$ be a finite $2'$-elusive biquasiprimitive permutation group with point stabiliser $H=G_{\a}$ and suppose that $G^+$ acts faithfully on its two orbits. Then $G$ has a unique minimal normal subgroup $N =T^k$, where $k \geqs 1$ and $T, N_{\a}$ are one of the following:
\begin{itemize}\addtolength{\itemsep}{0.2\baselineskip}
\item[{\rm (i)}] $k=1$, $T=A_6$ and $N_{\alpha}=A_5$;
\item[{\rm (ii)}] $T={\rm M}_{11}$ and $N_{\alpha}=\PSL_2(11)^k$;
\item[{\rm (iii)}] $T={}^2F_4(2)'$ and $N_{\alpha}=\PSL_2(25)^k$;
\item[{\rm (iv)}] $T=\PSL_2(p)$ and $(C_p{:}C_r)^k\leqslant N_{\alpha} <(C_p{:}C_{(p-1)/2})^k$ where $p$ is a Mersenne prime and $r$ is the product of the distinct prime divisors of $(p-1)/2$.
\end{itemize}
\end{lemma}

\begin{proof}
Let $N$ be a minimal normal subgroup of $G$. By Lemma \ref{lem:inGplus}, $N\leqslant G^+$ and since $G$ is biquasiprimitive, $N$ acts transitively on both $\Delta_1$ and $\Delta_2$. Moreover, since $G$ is transitive and we are assuming that $G^+$ acts faithfully on $\Delta_1$ and $\Delta_2$, it follows that $N^{\Delta_1}\cong N\cong N^{\Delta_2}$.

We claim that $N$ is the unique minimal normal subgroup of $G$. To see this, suppose that $M$ is another minimal normal subgroup of $G$, so $N^{\Delta_1}$ and $M^{\Delta_1}$ are both transitive normal subgroups of $(G^+)^{\Delta_1}$. Since $N^{\Delta_1} \cap M^{\Delta_1}=1$ it follows that $[N^{\Delta_1}, M^{\Delta_1}]=1$, so \cite[Theorem 4.2A]{DM} implies that $N^{\Delta_1}$ and $M^{\Delta_1}$ are regular on $\Delta_1$ and $N\cong M\cong T^k$ for some finite nonabelian simple group $T$ and positive integer $k$. Similarly, $N$ and $M$ act faithfully and regularly on $\Delta_2$ and thus every element of odd prime order in $N$ is a derangement on $\Omega$. This is a contradiction since $G$ is $2'$-elusive, so $N$ is the unique minimal normal subgroup of $G$ as claimed. Write $N=T^k$, where $T$ is simple and $k \geqs 1$. 

If $N$ is abelian then it is semiregular on $\Omega$ with two orbits, so $T=C_2$ is the only possibility since $G$ is $2'$-elusive. But this implies that $|\Omega|=2^{k+1}$, which is a contradiction since $|\Omega|$ is divisible by an odd prime. Therefore, $T$ is nonabelian. Write $N=T_1\times \cdots \times T_k$ with $T_i\cong T$, and let $\pi_i:N \to T_i$ be the $i$-th projection map. Note that $G \leqslant {\rm Aut}(N) = {\rm Aut}(T) \Wr S_k$.

Fix $\alpha\in\Delta_1$ and set $H=G_{\a}$. Now each $n \in N$ of odd prime order fixes an element of $\Omega$ and is therefore $G$-conjugate to an element of $N_\alpha$. Thus every $\Aut(N)$-class of elements of odd prime order in $N$ meets $N_\alpha$. Let $t\in T$ be an element of odd prime order. Then $(t,\ldots,t) \in N$ is $\Aut(N)$-conjugate to an element of $N_{\alpha}$, so for each $i\in\{1,\ldots,k\}$ we see that $\pi_i(N_{\alpha})$ meets every $\Aut(T_i)$-class of elements of odd prime order in $T_i$. Hence either $\pi_i(N_{\alpha})=T_i$, or $(T_i,\pi_i(N_\alpha))$ is given by Theorem \ref{thm:autT}(i). Similarly, $(t,1,\ldots,1) \in N$ is also $\Aut(N)$-conjugate to an element of $N_\alpha$, so 
\begin{equation}\label{e:nal}
1 \neq N_{\alpha} \cap T_{\ell} \normeq \pi_{\ell}(N_\alpha) \leqs T_{\ell}
\end{equation}
for some $\ell \in \{1,\ldots,k\}$. Since $N$ is faithful on $\Delta_1$ it follows that $T_{\ell} \not\leqs N_\alpha$ and so the fact that $T_{\ell}$ is simple implies that $\pi_{\ell}(N_\alpha)\neq T_{\ell}$. In particular, 
$(T_{\ell},\pi_{\ell}(N_\alpha))$ is one of the cases in Theorem \ref{thm:autT}(i). Set $R = \pi_{\ell}(N_\alpha)$.

The transitivity of $N$ on $\Delta_1$ implies that $G^+=NH$, so $G^+$ and $H$ have the same orbits on $\{T_1, \ldots, T_k\}$. The minimality of $N$ implies that $G$ acts transitively on this set, so $G^+$ is either transitive or has two equal sized orbits (since $|G:G^+|=2$). Let $\mathcal{O}_1$ be the orbit of $G^+$ on $\{T_1, \ldots, T_k\}$ containing $T_{\ell}$ (where $\ell$ is the integer in \eqref{e:nal}). Without loss of generality we may assume that 
$\{T_1,\ldots,T_{\lfloor k/2 \rfloor}\}\subseteq \mathcal{O}_1$. Note that $\pi_i(N_\alpha)\cong \pi_{\ell}(N_\alpha)$  and $N_\alpha \cap T_i\cong N_\alpha \cap T_{\ell}$ for all $i\in\mathcal{O}_1$. We now consider two cases.

\vs

\noindent \emph{Case 1.} $T \ne \PSL_2(p)$

\vs

Suppose first that $T\neq \PSL_2(p)$, in which case $\pi_{\ell}(N_\alpha)$ is simple (see Table \ref{tab:alloddautT}). In view of \eqref{e:nal}, it follows that $N_\alpha\cap T_{\ell}= \pi_{\ell}(N_\alpha) = R$. We claim that $N_\alpha \cong R^k$. This is clear if $\mathcal{O}_1=\{T_1,\ldots,T_k\}$, so let us assume that $k$ is even and 
$\mathcal{O}_1=\{T_1,\ldots,T_{k/2}\}$, so $R_1 \times \cdots \times R_{k/2} \leqslant N_\alpha$, where $R_i = \pi_i(N_{\a})\cong R$ and $R_i < T_i$ for each $i \in \{1, \ldots, k/2\}$. 

Let $t \in T$ be an element of odd prime order and set $g=(t,1,\ldots,1, t)\in N$. Since
$$\{T_1,\ldots,T_{k}\} = \{T_1,\ldots,T_{k/2}\} \cup \{T_{k/2+1},\ldots,T_{k}\}$$
is a $G$-invariant partition, it follows that every $G$-conjugate of $g$ has precisely one nontrivial entry in the first $k/2$ coordinates and precisely one nontrivial entry in the last $k/2$ coordinates. Since $G$ is $2'$-elusive, $g$ is conjugate to an element of $N_{\alpha}$. By multiplying this conjugate by an appropriate element of $R_1\times \cdots \times R_{k/2}$, we deduce that $N_\alpha$ contains an element with precisely one nontrivial entry, which occurs in the last $k/2$ coordinates. Hence $1\neq N_\alpha\cap T_i\normeq \pi_i(N_\alpha)$ for all $i\in\{k/2+1,\ldots,k\}$. By arguing as above we deduce that $\pi_i(N_\alpha)\cong R$ and thus $N_{\alpha}\cong R^k$. This justifies the claim.

We now consider the possibilities for $T$ arising in Theorem \ref{thm:autT}(i). If 
$T=\Omega_{8}^+(2)$ or $\POmega_{8}^+(3)$ then the proof of  \cite[Proposition 4.6]{GX}  produces a derangement of order $5$ in $N$. Similarly, if $T=A_6$ and $k\geqs 2$ then the same proof gives a derangement of order $3$ (if $k=1$ then case (i) holds). If $T={\rm M}_{11}$ or ${}^2F_4(2)'$ then we are in case (ii) or (iii), respectively.

\vs

\noindent \emph{Case 2.} $T = \PSL_2(p)$

\vs

Finally, let us assume that $T=\PSL_2(p)$, so $p$ is a Mersenne prime (see Table \ref{tab:alloddautT} and Remark \ref{r:thm}). We have seen that for each element $t\in T$ of odd prime order,  $(t,1,\ldots,1)$ is $G$-conjugate to an element of $N_{\alpha}$, and that the unique nontrivial entry of this element lies in $\mathcal{O}_1\subseteq \{T_1,\ldots,T_{k}\}$. Since $H$ acts transitively on the set $\{N_\alpha\cap T_i \mid T_i \in\mathcal{O}_1\}$, it follows that $N_\alpha \cap T_i$ meets each ${\rm Aut}(T_i)$-class of elements of odd prime order in $T_i$, for all $i\in \mathcal{O}_1$. This immediately implies that (iv) holds if $\mathcal{O}_1=\{T_1,\ldots,T_k\}$. 

To complete the proof, we may assume $k$ is even and $\mathcal{O}_1 = \{T_{1},\ldots, T_{k/2}\}$. The above argument shows that $Q_1\times\cdots\times Q_{k/2}\leqslant N_\alpha$, where $Q_i\cong C_p{:}C_r$ for all $i$ (here $r$ is the product of the distinct prime divisors of $(p-1)/2$). Moreover, any $n \in N_\alpha$ of odd prime order projects onto an element of $Q_i$ for each $i\in\{1,\ldots,k/2\}$. 

Let $t \in T$ be an element of odd prime order and set $g=(t,1,\ldots,1, t)\in N$. As observed above, each $G$-conjugate of $g$ has precisely one nontrivial entry in the first $k/2$ coordinates and one in the last $k/2$ coordinates. An appropriate $G$-conjugate of $g$ is contained in $N_\alpha$, which we can multiply by an element of $Q_1\times\cdots\times Q_{k/2} \leqs N_{\a}$ to obtain an element of odd prime order in $N_{\a}$ with precisely one nontrivial entry in the $i$-th coordinate for some $i\in\{k/2+1,\ldots,k\}$. Therefore $1\neq N_\alpha\cap T_i \normeq \pi_i(N_\alpha)$. Since $T_i$ is simple and $\pi_i(N_\alpha)$ meets every 
$\Aut(T_i)$-class of elements of odd prime order in $T_i$, it follows that $\pi_i(N_\alpha)\leqslant C_p{:}C_{(p-1)/2}$. Moreover, we also see that $N_\alpha\cap T_i$ meets every 
$\Aut(T_i)$-class of elements of odd prime order in $T_i$, so $C_p{:}C_r\leqslant N_\alpha\cap T_i$. Therefore,
$$C_p{:}C_r\leqslant N_\alpha\cap T_i \leqs \pi_i(N_\alpha)\leqslant C_p{:}C_{(p-1)/2}$$
for all $i\in\{k/2+1,\ldots,k\}$, and we conclude that (iv) holds.  
\end{proof}

\begin{theorem}\label{thm:faith}
Let $G \leqs {\rm Sym}(\Omega)$ be a finite $2'$-elusive biquasiprimitive permutation group with point stabiliser $H=G_{\a}$ and socle $N$. Let $K \leqs S_k$ be the transitive group induced by $G$ on the set of $k$ simple direct factors of $N$ and let $K^+ \leqs K$ be the group induced by $G^+$. Assume that $G^+$ acts faithfully on its two orbits. Then one of the following holds:
\begin{itemize}\addtolength{\itemsep}{0.2\baselineskip}
\item[{\rm (i)}] $(G,H) = ({\rm M}_{10},A_5)$ or $({\rm Aut}(A_6),S_5)$;
\item[{\rm (ii)}]  $G={\rm M}_{11}\Wr K$, $H=\PSL_2(11)\Wr K^+$ and $|K:K^{+}|=2$;
\item[{\rm (iii)}] $N=({}^2F_4(2)')^k\normeq G\leqslant {}^2F_4(2)\Wr K$ and $H=N_{G^+}(N_\alpha)$, where $N_{\a}=\PSL_2(25)^k$;
\item[{\rm (iv)}] $N=\PSL_2(p)^k\normeq G\leqslant\PGL_2(p)\Wr K$,  
$$(C_p{:}C_r)^k\leqslant N_{\alpha} <(C_p{:}C_{(p-1)/2})^k$$
and
$$H < N_{G^+}(N_\alpha) = G^+\cap ((C_p{:}C_{p-1}) \Wr K),$$
where $p$ is a Mersenne prime and $r$ is the product of the distinct prime divisors of $(p-1)/2$. 
\end{itemize}
Moreover, each group $G$ in (i), (ii) and (iii) is $2'$-elusive and biquasiprimitive.
\end{theorem}

\begin{proof}
By Lemma \ref{lem:NGplusfaith}, $N$ is the unique minimal normal subgroup of $G$ and we may write $N=T^k$ with $k \geqs 1$, where the possibilities for $T$ and $N_{\a}$ are described in the lemma. Note that $G\leqslant \Aut(N) = \Aut(T) \Wr S_k$.

First assume that $T=A_6$ and $N_{\alpha}=A_5$, so $G\leqslant \Aut(A_6)$. Since $G$ is $2'$-elusive, each element in $N$ of odd prime order is $G$-conjugate to an element of $N_{\alpha}$, so $G$ must contain an outer automorphism of $S_6$. Therefore $G={\rm M}_{10}$ or $\Aut(A_6)$, and $N$ has two orbits of size $6$. It follows that $(G,H) = ({\rm M}_{10},A_5)$ or $({\rm Aut}(A_6),S_5)$, as in (i). It is easy to check that $G$ is indeed $2'$-elusive and biquasiprimitive in both cases. 

Next assume that $T={\rm M}_{11}$ and $N_{\alpha}=\PSL_2(11)^k$. 
Since $\Aut(T)=T$ it follows that $G=T \Wr K$ for some transitive subgroup $K \leqs S_k$ (the transitivity of $K$ follows from the minimality of $N$). Moreover, $G^+=T \Wr K^+$ and $H=\PSL_2(11) \Wr K^+$, where $|K:K^+|=2$. This is case (ii) in the statement of the theorem. 

We claim that every group $G$ as in (ii) is $2'$-elusive and biquasiprimitive. First observe that $N=T^k$ is the unique minimal normal subgroup of $G$ and $N$ has two orbits on $\Omega$. Therefore, if $M$ is any nontrivial normal subgroup of $G$ then $N \leqs M$, so $M$ has at most two orbits on $\Omega$ and thus $G$ is biquasiprimitive. To see that $G$ is $2'$-elusive, let $\Delta$ denote the set of right cosets of ${\rm PSL}_{2}(11)$ in $T$. Since $N_{\alpha}=\PSL_{2}(11)^k$ and $G^+=T \Wr K^+$ acts transitively and faithfully on $\Delta_1$, we may identify $\Delta_1$ with the Cartesian product $\Delta^k$ so that $G^+$ acts on $\Delta^k$ with its standard product action (see Lemma \ref{lemPA}). The action of $T$ on $\Delta$ is elusive, so \cite[Theorem 4.1(e)]{CGJKKMN} implies that the action of $G^+$ on $\Delta_1$ is also elusive. Therefore, each $g \in G^+$ of prime order has fixed points on $\Delta_1$, and hence on $\Omega$. Since every element in $G$ of odd prime order lies in $G^+$, we deduce that $G$ is indeed $2'$-elusive.

Next suppose that $T={}^2F_4(2)'$ and $N_{\alpha}=\PSL_2(25)^k$. The minimality of $N$ implies that $G$ induces a transitive group $K \leqs S_k$ on the set of $k$ simple direct factors of $N$, so $N\normeq G\leqslant {}^2F_4(2)\Wr K$ (note that $\Aut({}^2F_4(2)')={}^2F_4(2) = {}^2F_4(2)'.2$) and $G/N$ is a subgroup of $C_2\Wr K$ that projects onto $K$. Note that $H \leqslant G^+$ and $G^+/N$ is an index-two subgroup of $G/N$. Since $G^+$ acts transitively and faithfully on $\Delta_1$, by Lemma \ref{lemPA} we may identify $\Delta_1$ with $\Delta^k$, where $\Delta$ is the set of right cosets of $\PSL_2(25)$ in $T$, so that $G^+$ acts on $\Delta^k$ via the usual product action. In particular, $H=N_{G^+}(N_\alpha)$ as in part (iii). By arguing as above, we see that every group $G$ as in (iii) is biquasiprimitive. Also note that the action of $T$ on $\Delta$ is $2'$-elusive (see Theorem \ref{t:prim2dash}), so Lemma \ref{lem:wr} implies that the action of $G^+$ on $\Delta_1$ is also $2'$-elusive and we conclude that $G$ is $2'$-elusive as above.

Finally, let us assume that $T=\PSL_2(p)$ and 
$$(C_p{:}C_r)^k\leqslant N_{\alpha} <(C_p{:}C_{(p-1)/2})^k$$ 
where $p$ is a Mersenne prime and $r$ is the product of the distinct prime divisors of $(p-1)/2$.
Here $N\normeq G\leqslant \PGL_2(p)\Wr K$  for some transitive subgroup $K \leqs S_k$, as in case (iv), and $G/N$ is a subgroup of $C_2\Wr K$ that projects onto $K$. In addition, we note that $H<N_{G^+}(N_\alpha)$ since $N_\alpha \normeq H$ and 
$|G^+:N_{G^+}(N_\alpha)|$ is a power of $2$.
\end{proof}

\begin{remark}
Let $G$ be a group as in case (iii) of Theorem \ref{thm:faith}. In general, there is more than one possibility for $H = G_{\a}$ with the desired property that the action of $G$ on $G/H$ is $2'$-elusive and biquasiprimitive. For example, if $G = \langle ({}^2F_4(2)')^k, (g,\ldots,g)\rangle {:} S_k$, where ${}^2F_4(2)=\langle {}^2F_4(2)',g \rangle$ and $k \geqs 2$, then we can take $H=N_{G^+}(N_\alpha)$ with 
$$G^{+} = \mbox{$\langle ({}^2F_4(2)')^k, (g,\ldots,g)\rangle {:}A_k$ or ${}^2F_4(2)' \Wr S_k$}$$
and $N_{\a} = {\rm PSL}_{2}(25)^k$.
\end{remark}

\begin{remark}\label{r:exd}
Examples do occur in case (iv) of Theorem \ref{thm:faith}. To see this, fix a Mersenne prime $p$ such that $r < (p-1)/2$. Then the almost simple group $G=\PGL_2(p)$ with $H=C_p{:}C_{r}$ as in Theorem \ref{t:2dashas} is biquasiprimitive and $2'$-elusive (note that $G^+=\PSL_2(p)$). Similarly, if $k \geqs 2$ then we can take $G=\PSL_2(p)\Wr S_k$ and $H=(C_p{:}C_r)\Wr A_k$ (here $G^+=\PSL_2(p)\Wr A_k$). However, it is important to note that not all of the groups arising in (iv) are both biquasiprimitive and $2'$-elusive. For instance, we highlight the following examples: 
\begin{enumerate}
\item[(a)] $G=\PSL_2(p) \Wr S_3$ with $H= (C_p{:}C_r)^3$ is neither $2'$-elusive (it has derangements of order three) nor biquasiprimitive ($N=\PSL_2(p)^3$ has 6 orbits on $\Omega$),
\item[(b)] $G=\PSL_2(p)\Wr C_4$ with $H= (C_p{:}C_r)^4$ is $2'$-elusive but not biquasiprimitive. 
\end{enumerate}  
\end{remark}

\subsection{$G^{+}$ is not faithful on both orbits}\label{ss:notfaith}

To complete the proof of Theorem \ref{t:bqprim2dash}, we may assume that $G^+$ is not faithful on at least one of its two orbits $\Delta_1$ and $\Delta_2$ on $\Omega$. We begin with a lemma that describes the structure of ${\rm soc}(G)$ and ${\rm soc}(G)_{\a}$.

\begin{lemma}\label{lem:NG+unfaith}
Let $G \leqs {\rm Sym}(\Omega)$ be a finite $2'$-elusive biquasiprimitive permutation group with point stabiliser $H=G_{\a}$ and assume that $G^+$ is not faithful on at least one of its orbits. Then $G$ has a unique minimal normal subgroup $N=T^k$, where $k \geqs 2$ is even and $T, N_{\a}$ are one of the following: 
\begin{itemize}\addtolength{\itemsep}{0.2\baselineskip}
\item[{\rm (i)}] $T={\rm M}_{11}$ and $N_\alpha = \PSL_2(11)^{k/2} \times {\rm M}_{11}^{k/2}$; 
\item[{\rm (ii)}] $T={}^2F_4(2)'$ and $N_\alpha = \PSL_2(25)^{k/2}\times ({}^2F_4(2)')^{k/2}$;
\item[{\rm (iii)}] $T=\PSL_{2}(p)$ and 
$$(C_p{:}C_r)^{k/2}\times \PSL_2(p)^{k/2} \leqslant N_\alpha < (C_p{:}C_{(p-1)/2})^{k/2}\times \PSL_2(p)^{k/2},$$
where $p$ is a Mersenne prime and $r$ is the product of the distinct prime divisors of $(p-1)/2$.
\end{itemize}
\end{lemma}

\begin{proof}
We adapt the proof of \cite[Lemma 4.7]{GX}. Without loss of generality, we may assume that the action of $G^{+}$ on $\Delta_1$ is not faithful. Let $M_1$ be a minimal normal subgroup of $G^+$ contained in the kernel of the action of $G^+$ on $\Delta_1$. Fix an element $g\in G\setminus G^+$ and observe that $M_1^g$ is a minimal normal subgroup of $G^+$ contained in the kernel of the action of $G^+$ on $\Delta_2$ (in particular, $G^+$ is not faithful on either orbit). Since $G$ is faithful on $\Omega$, it follows that $M_1\cap M_1^g=1$, $M_1$ is faithful on $\Delta_2$ and $M_1^g$ is faithful on $\Delta_1$. In addition, $M_1$ acts transitively on $\Delta_2$, and $M_1^g$ acts transitively on $\Delta_1$ (since $G$ is biquasiprimitive). Since $g^2\in G^+$ we deduce that $(M_1^g)^g=M_1$ and so $N=M_1\times M_1^g$ is a minimal normal subgroup of $G$.  Moreover, if $h\in M_1$ is a derangement on $\Delta_2$, then $(h,h^g)\in N$ is a derangement on $\Omega$. Therefore $M_1$ is $2'$-elusive on $\Delta_2$, so the possibilities for $M_1$ and $(M_1)_\alpha$ are given by Lemma \ref{l:socle} (where $\a \in \Delta_2$). It follows that $N=T^k$ for some even integer $k \geqs 2$, and 
$$N_\alpha=(M_1)_\alpha\times (M_1)^g\cong (M_1)_\alpha\times T^{k/2}.$$ 
Therefore, to complete the proof of the lemma it remains to show that $N$ is the unique minimal normal subgroup of $G$. Set $L=(G^+)^{\Delta_2}$ and note that 
$$\soc(G)\leqslant \soc(L)\times \soc(L)$$
by \cite[Lemma 3.2(a)]{Praeger}.

First assume that $T={\rm M}_{11}$ or ${}^2F_4(2)'$. Then $(M_1)_\alpha$ is self-normalising in $M_1$, so \cite[Theorem 4.2A]{DM} implies that $C_{\Sym(\Delta_2)}(M_1)=1$ and thus $\soc(L)=M_1$. We conclude that $N$ is the unique minimal normal subgroup of $G$.

Finally, let us assume that $T=\PSL_2(p)$ with $p$ a Mersenne prime. As usual, let $r$ be the product of the distinct prime divisors of $(p-1)/2$. Then   
$$(C_p{:}C_r)^{k/2} \leqslant (M_1)_\alpha \leqslant (C_p{:}C_{(p-1)/2})^{k/2}$$ 
and we note that $C_p{:}C_{r} \normeq C_p{:}C_{(p-1)/2}$ and $C_p{:}C_{(p-1)/2}$ is a maximal subgroup of $T$. In particular, $N_{M_1}((M_1)_\alpha)=(C_p{:}C_{(p-1)/2})^{k/2}$ has odd order. Therefore,  \cite[Theorem 4.2A]{DM} implies that $C_{L}(M_1)$ has odd order and is semiregular on $\Delta_2$.

If $C_L(M_1)=1$ then $\soc(L)=M_1$ and so $N$ is the unique minimal normal subgroup of $G$. Now assume $C_{L}(M_1)\neq 1$. Let $J$ be a minimal normal subgroup of $L$ that is contained in $C_{L}(M_1)$. Since $C_L(M_1)$ has odd order, $J$ is elementary abelian. However, $|\Delta_2|$ is divisible by $p+1$, which is a power of 2, and so $J$ is intransitive on $\Delta_2$. In particular, $L$ is not quasiprimitive on $\Delta_2$. Moreover, since $G^+$ is not faithful on its orbits, \cite[Lemma 3.5]{Praeger} implies that the structure of $G$ is as in case (b) of \cite[Theorem 1.1]{Praeger}. In particular, $M_1$ is the unique transitive minimal normal subgroup of $L$ and $\soc(G)=M_1\times (M_1)^g=N$, so $N$ is the unique minimal normal subgroup of $G$. 
\end{proof}

We are now in a position to complete the proof of Theorem \ref{t:bqprim2dash}. In the statement and proof of Theorem \ref{t:bi2}, we will write $R$ for the almost simple maximal subgroup $\PSL_2(25).2_3<{}^2F_4(2)$ (this is a nonsplit extension). 

\begin{theorem}\label{t:bi2}
Let $G \leqs {\rm Sym}(\Omega)$ be a finite $2'$-elusive biquasiprimitive permutation group with point stabiliser $H=G_{\a}$ and socle $N$. Let $K \leqs S_k$ be the transitive group induced by $G$ on the set of $k$ simple direct factors of $N$ and let $K^+ \leqs K$ be the group induced by $G^+$. Assume that $G^+$ is not faithful on at least one of its orbits. Then $k$ is even, $K^+$ is intransitive, $|K:K^+|=2$ and one of the following holds:
\begin{itemize}\addtolength{\itemsep}{0.2\baselineskip}
\item[{\rm (i)}] $G={\rm M}_{11}\Wr K$ and $H=(\PSL_2(11)^{k/2} \times {\rm M}_{11}^{k/2}){:}K^+$;

\item[{\rm (ii)}] $N = ({}^2F_4(2)')^k\normeq G\leqslant {}^2F_4(2)\Wr K$, $N_{\a} = {\rm PSL}_{2}(25)^{k/2} \times ({}^2F_4(2)')^{k/2}$ and  
$$H = N_{G^+}(N_{\a}) = G^+\cap (R^{k/2}\times {}^2F_4(2)^{k/2}){:}K^+$$
with $G^+ = G \cap ({}^2F_4(2)\Wr K^+)$; 

\item[{\rm (iii)}] $N = \PSL_{2}(p)^k \normeq G \leqslant \PGL_2(p)\Wr K$, 
$$(C_p{:}C_r)^{k/2} \times \PSL_2(p)^{k/2} \leqslant N_{\alpha} <(C_p{:}C_{(p-1)/2})^{k/2}\times \PSL_2(p)^{k/2}$$
and
$$H < N_{G^+}(N_\alpha)=G^+\cap ((C_p{:}C_{p-1})^{k/2}\times \PGL_2(p)^{k/2}){:}K^+,$$
where $G^+ = G \cap (\PGL_2(p)\Wr K^+)$ and $p$ is a Mersenne prime.
\end{itemize}
Moreover, each group $G$ in (i) and (ii) is $2'$-elusive and biquasiprimitive.
\end{theorem}

\begin{proof}
By Lemma \ref{lem:NG+unfaith}, $G$ has a unique minimal normal subgroup $N=T^k$, where $k \geqs 2$ is even and the possibilities for $T$ and $N_{\a}$ are described in the lemma. In particular, we note that $G\leqslant \Aut(N) = \Aut(T) \Wr S_k$. Write $N = T_1 \times \cdots \times T_k$ with $T_i \cong T$ for each $i$, and let $K \leqs S_k$ be the permutation group induced by the conjugation action of $G$ on $\{T_1, \ldots, T_k\}$. Note that $K$ is transitive since $N$ is minimal. Moreover, $G\leqslant\Aut(T)\Wr K$.

Now $G^+=NH$ and $H \leqs N_{G^+}(N_{\alpha})$, so $G^+$ has two orbits on $\{T_1, \ldots, T_k\}$ and it induces an intransitive index-two subgroup $K^+<K$. Note that $G^+=G\cap (\Aut(T)\Wr K^+)$.  We may assume that the orbits of $K^+$ are $\{T_1, \ldots, T_{k/2}\}$ and $\{T_{k/2+1}, \ldots, T_{k}\}$. We now consider the three cases arising in Lemma \ref{lem:NG+unfaith}.

First assume that $T={\rm M}_{11}$ and $N_{\alpha}=\PSL_2(11)^{k/2}\times {\rm M}_{11}^{k/2}$. Since $\Aut({\rm M}_{11})={\rm M}_{11}$ it follows that $G={\rm M}_{11}\Wr K$, so $G^+={\rm M}_{11}\Wr K^+$ and since $G^+=NH$ we deduce that $H=(\PSL_2(11)^{k/2}\times {\rm M}_{11}^{k/2}){:}K^+$ as in case (i). Now each $g \in G^+$ of prime order is conjugate to an element of $\PSL_2(11)\Wr K^+$ (this follows from \cite[Theorem 4.1(e)]{CGJKKMN}), which is contained in $H$, so any  group $G$ of this form is $2'$-elusive since every element in $G$ of odd prime order is contained in $G^+$. In addition, $G$ is biquasiprimitive since every nontrivial normal subgroup of $G$ contains $N$, which is transitive on $\Delta_1$ and $\Delta_2$.

Next assume $T={}^2F_4(2)'$ and $N_{\a} = {\rm PSL}_{2}(25)^{k/2} \times ({}^2F_4(2)')^{k/2}$. Here $G\leqslant {}^2F_4(2)\Wr K$, $G^+ = G \cap ({}^2F_4(2)\Wr K^+)$ and 
$$H \leqslant N_{G^+}(N_\alpha)=G^+\cap (R^{k/2}\times {}^2F_4(2)^{k/2}){:}K^+$$ 
as in (ii). 
Since $\PSL_2(25)$ is a maximal subgroup of ${}^2F_4(2)'$, and $G^+$ acts transitively on $\{T_1, \ldots, T_{k/2}\}$, it follows that $G^+$ acts primitively on $\Delta_1$ and $\Delta_2$, inducing a subgroup of ${}^2F_4(2) \Wr S_{k/2}$ on each orbit. Therefore $H$ is a maximal subgroup of $G^+$ and thus $H=N_{G^+}(N_\alpha)$. As in the previous case, any such group $G$ is biquasiprimitive. Moreover, every element in $G$ of odd prime order is contained in $G^+$, and Theorem \ref{t:prim2dash} implies that every element of odd prime order in $G^+$ is conjugate to an element of $\PSL_2(25) \Wr K^+$, which is contained in $H$. We conclude that $G$ is $2'$-elusive.

Finally, let us assume that $T=\PSL_{2}(p)$ with $p$ a Mersenne prime. Here we have $G\leqslant \PGL_2(p)\Wr K$, $G^+ = G \cap (\PGL_2(p)\Wr K^+)$ and 
$$H < N_{G^+}(N_\alpha)=G^+\cap (C_p{:}C_{p-1})^{k/2}\times \PGL_2(p)^{k/2}){:}K^+$$
as in (iii). Note that $H$ is a proper subgroup of $N_{G^+}(N_\alpha)$ since $|G:N_{G^+}(N_\alpha)|$ is a power of $2$ and $|\Omega| = |G:H|$ is divisible by an odd prime.
\end{proof}

\vs

This completes the proof of Theorem \ref{t:bqprim2dash}.

\begin{remark}\label{r:bqex}
Examples of $2'$-elusive groups in case (iii) do exist. For instance, if $p$ is a Mersenne prime with $r<(p-1)/2$, then we can take
$$G=\PSL_2(p) \Wr (S_{k/2}\Wr S_2),\;\; H= ((C_p{:}C_r)^{k/2}\times \PSL_2(p)^{k/2}) {:} S_{k/2}^2,$$ 
where $S_{k/2}\Wr S_2$ acts imprimitively on $k$ points. Notice that $|G:NH|=2$ in these examples (where $N = \soc(G)$), so $G$ is indeed biquasiprimitive.
\end{remark}

\section{Arc-transitive graphs of prime valency}\label{s:graphs}

In this final section we will use Theorems \ref{t:prim2dash},  \ref{t:qprim2dash} and \ref{t:bqprim2dash} to determine the $2'$-elusive quasiprimitive and biquasiprimitive groups with a prime subdegree. As explained in the Introduction, this is the main step in the proof of Theorem \ref{t:graph}, which we anticipate will play a key role in the proof of the Polycirculant Conjecture for arc-transitive graphs of valency $2p$, where $p$ is an odd prime. We start by recalling some standard terminology.

Let $G \leqs {\rm Sym}(\Omega)$ be a finite transitive permutation group and let $\alpha\in\Omega$. Recall that the orbits of $G_\alpha$ on $\Omega \setminus \{\alpha\}$ are 
called \emph{suborbits}, and the lengths of these orbits are the \emph{subdegrees} of $G$. It is well known that there is a one-to-one correspondence between the set of suborbits of $G$ and the set of digraphs with vertex set $\Omega$ on which $G$ acts arc-transitively. 
More precisely, the suborbit corresponding to a given digraph $\Gamma$ is the set $\Gamma^+(\alpha)$ of out-neighbours of $\alpha$ in $\Gamma$. On the other hand, if $\b^{G_{\a}}$ is the suborbit containing $\b$ then the corresponding digraph on $\Omega$ has arc-set 
$$(\alpha,\beta)^G = \{(\alpha^g,\beta^g)\mid g\in G\},$$
which is simply the orbit of $(\a,\b)$ with respect to the natural action of $G$ on $\Omega \times \Omega$. 
Such an arc-set is called an \emph{orbital} of $G$, and the corresponding digraph is referred to as an \emph{orbital digraph}. Further, we say that the suborbit $\b^{G_{\a}}$, and also the orbital $(\alpha,\beta)^G$, is \emph{self-paired} if there is an element $g\in G$ that interchanges $\alpha$ and $\beta$. In this situation, $(\omega_1,\omega_2)\in (\alpha,\beta)^G$ if and only if 
$(\omega_2,\omega_1)\in (\alpha,\beta)^G$, in which case the corresponding digraph $\Gamma$ is a graph since we can ignore the directions on the edges (note that $\Gamma$ is $|\b^{G_{\a}}|$-regular).

Note that if $g \in G$ interchanges $\alpha$ and $\beta$ then it lies in $N_G(G_\alpha\cap G_\beta)$, but not in $G_\alpha$. Also note that if $\alpha^g=\beta$ then $G_{\alpha,\beta}=G_\alpha\cap (G_\alpha)^g$ and $|\beta^{G_\alpha}|=|G_\alpha:G_\alpha \cap (G_\alpha)^g|$.
We will need the following two lemmas (the proofs are easy exercises).

\begin{lemma}\label{l:con}
Let $G \leqs {\rm Sym}(\Omega)$ be a finite transitive permutation group and let $\b^{G_{\a}}$ be a self-paired suborbit of $G$. Then the corresponding orbital graph is connected if and only if $G=\langle G_\alpha,g\rangle$ for each $g\in G$ that interchanges $\alpha$ and $\beta$.
\end{lemma}

\begin{lemma}\label{lem:imprimsubdeg}
Let $G\leqslant\Sym(\Omega)$ be a transitive imprimitive permutation group with system of imprimitivity $\mathcal{P}$. Let $\alpha,\omega\in\Omega$ and let $A\in\mathcal{P}$ be the block containing $\omega$. Then the following hold:
\begin{itemize}\addtolength{\itemsep}{0.2\baselineskip}
\item[{\rm (i)}] $|A^{G_{\alpha}}|$ divides $|\omega^{G_{\alpha}}|$;
\item[{\rm (ii)}] If $\alpha \in A$ then the orbital digraph corresponding to $\omega^{G_{\a}}$ is disconnected.
\end{itemize}
\end{lemma}

In order to prove Theorem \ref{t:graph}, we need to determine the $2'$-elusive quasiprimitive and biquasiprimitive  groups  with a self-paired suborbit of prime length which satisfies the connectedness condition in Lemma \ref{l:con}. To get started, in the next lemma we determine the subdegrees of the relevant almost simple groups (see Theorem \ref{t:2dashas}). Note that in the statement of the lemma, we write $\ell^k$ to denote that $\ell$ occurs as a subdegree $k$ times.   

\begin{lemma}\label{lem:smallsubdegs}
Let $G \leqs {\rm Sym}(\Omega)$ be a finite $2'$-elusive quasiprimitive or biquasiprimitive almost simple  group with point stabiliser $H$. Let $\mathcal{S}$ be the multiset of subdegrees of $G$. Then $G$ has a prime subdegree if and only if one of the following holds:
\begin{itemize}\addtolength{\itemsep}{0.2\baselineskip}
\item[{\rm (i)}] $(G,H) = ({\rm M}_{11}, {\rm PSL}_{2}(11))$, $\mathcal{S} = \{1, 11\}$;
\item[{\rm (ii)}] $(G,H) = ({\rm M}_{10}, A_5)$ or $(\Aut(A_6),S_5)$, $\mathcal{S} = \{1,5,6\}$; 
\item[{\rm (iii)}] $(G,H) = (\PSL_2(p), C_p{:}C_s)$ or $(\PGL_2(p), C_p{:}C_{2s})$, $\mathcal{S} = \{1^{(p-1)/2s}, p^{(p-1)/2s}\}$, $p$ is a Mersenne prime, $C_r \leqs C_s < C_{(p-1)/2}$ and $r$ is the product of the distinct prime divisors of $(p-1)/2$;    
\item[{\rm (iv)}]  
$(G,H) = (\PGL_2(p), C_p{:}C_s)$, $\mathcal{S} =\{ 1^{(p-1)/s}, p^{(p-1)/s}\}$, where $p,s$ are as in part (iii).
\end{itemize}
Moreover, if $\Delta$ is a suborbit of prime length and $\Gamma$ is the corresponding orbital digraph, then $\Delta$ is self-paired and $\Gamma$ is connected, with the exception of case (ii).
\end{lemma}
 
\begin{proof}
The first two cases can be easily checked using {\sc Magma} \cite{magma}. Note that in case (i), $\Gamma$ is the complete graph $K_{12}$, while for the suborbit of length $5$ in case (ii), $\Gamma$ is the disjoint union of two copies of $K_6$. Similarly, one checks that
\renewcommand{\arraystretch}{1.2}
$$\mathcal{S} = \left\{\begin{array}{ll}
\{1,78, 300^2, 325^2, 975\} & \mbox{$(G,H) = ({}^2F_4(2)',{\rm PSL}_{2}(25))$ or $({}^2F_4(2),{\rm PSL}_{2}(25).2_3)$} \\
\{1^2,78^2, 300^4, 325^4, 975^2\} & \mbox{$(G,H) = ({}^2F_4(2), {\rm PSL}_{2}(25))$}
\end{array}\right.$$
\renewcommand{\arraystretch}{1}
so no cases with ${\rm soc}(G) = {}^2F_4(2)'$ arise.

In view of Theorem \ref{t:2dashas}, we may assume that $\soc(G)= {\rm PSL}_{2}(p)$ with $p$ a Mersenne prime. First consider case (iii), with $(G,H) = ({\rm PSL}_{2}(p), C_p{:}C_s)$. Let $N=N_G(H)=C_p{:}C_{(p-1)/2}$. By \cite[Theorem 4.2A(i)]{DM}, $H$ has $|N_G(H):H|=(p-1)/2s$ fixed points on $\Omega$. Note that the action of $G$ on $G/N$ is $2$-transitive with degree $p+1$, so in this action $G$ has a unique suborbit of length $p$ and thus $N\cap N^g=C_{(p-1)/2}$ for all $g\in G\setminus N$. Fix an element $g\in G\setminus N$ and let $L$ be the unique subgroup of $N \cap N^g$ of order $s$, so $H \cap H^g \leqs L$. Now $N$ contains $p$ cyclic subgroups of order $s$, each of which is contained in $H$, so $L \leqs H$ and $L \leqs H^g$. We conclude that $H \cap H^g \cong C_s$ for all $g\in G\setminus N$, so $\mathcal{S} = \{1^{(p-1)/2s}, p^{(p-1)/2s}\}$ as claimed.

Let $\Delta = \beta^{G_{\a}}$ be a suborbit of length $p$, so $G_\alpha\cap G_\beta \cong C_s$.  Note that $G_\alpha\cap G_\beta$ fixes the $(p-1)/2s$ fixed points of $G_\alpha$, and it also fixes one point from each of the $(p-1)/2s$ orbits of $G_\alpha$ of length $p$.  Set $M=N_{G}(G_\alpha\cap G_\beta)=D_{p-1}$ and note that each element of $M$ maps $\alpha$ and $\beta$ to points fixed by $G_\alpha\cap G_\beta$.  Also note that $M \cap N = C_{(p-1)/2}$ transitively permutes the fixed points of $G_\alpha$ and the set of orbits of $G_\alpha$ of length $p$. Thus $M$ has at most two orbits on the set of $(p-1)/s$ fixed points of $G_\alpha\cap G_\beta$. In fact, $N$ is the stabiliser in $G$ of the set of fixed points of $G_\alpha$ (since $N$ is maximal in $G$), and thus $M$ is transitive on the set of fixed points of $G_\alpha\cap G_\beta$. In particular, each involution in $M\setminus (M\cap N)$ interchanges $\alpha$ with a fixed point of $M$ contained in a suborbit of length $p$. Therefore, there exists an element $g \in M\setminus (M\cap N)$ that interchanges $\a$ and $\b$, so $\Delta$ is self-paired. Moreover, since $g \notin N$ and $N$ is the unique maximal subgroup of $G$ containing $G_\alpha$, it follows that $G = \langle G_\alpha, g\rangle$ and thus the corresponding orbital graph is connected by Lemma \ref{l:con}.

A similar argument applies when $G = {\rm PGL}_{2}(p)$ in (iii) or (iv). We omit the details.
\end{proof}

\begin{lemma}\label{lem:prodactionsubdegs}
Let $G \leqs {\rm Sym}(\Omega)$ be a finite transitive permutation group such that 
$$N=\soc(L)^k\normeq G\leqslant L\Wr K,$$ 
$k \geqs 2$ and $G$ acts with its product action on $\Omega=\Delta^k$. Here $L\leqslant\Sym(\Delta)$ is transitive and almost simple, $K \leqs S_k$ is the group induced by $G$ on the set of $k$ simple direct factors of $N$. Assume that the following conditions are satisfied:
\begin{itemize}\addtolength{\itemsep}{0.2\baselineskip}
\item[{\rm (a)}] $N$ is transitive on $\Omega$;
\item[{\rm (b)}] The only element of $\Delta$ fixed by $\soc(L)_\delta$ is $\delta$;
\item[{\rm (c)}] Either $K$ is transitive, or it has two equal sized orbits on the set of $k$ simple direct factors of $N$. 
\end{itemize}
Then $G$ has a prime subdegree only if $k=2$, $K=1$ and $\soc(L)$ has a prime subdegree on $\Delta$. Moreover, if $G$ has a self-paired suborbit of prime length, then the corresponding orbital graph is disconnected.
\end{lemma}

\begin{proof}
Let $\alpha=(\delta,\ldots,\delta)\in\Omega$ and suppose $\omega^{G_\alpha}$ is a self-paired suborbit of prime length, where $\omega = (\omega_1, \ldots, \omega_k)\in\Omega\setminus\{\alpha\}$. Note that $\omega$ differs from $\alpha$ in at least one coordinate and since the only element of $\Delta$ fixed by $\soc(L)_\delta$ is $\delta$, it follows that $\alpha$ is the only element of $\Omega$ fixed by $N_\alpha=(\soc(L)_\delta)^k$. Therefore $|\omega^{N_{\alpha}}|>1$ divides $|\omega^{G_\alpha}|$ and thus $\omega^{N_\alpha}=\omega^{G_\alpha}$. 

For each element $\beta = (\b_1, \ldots, \b_k) \in \Omega$, we define the \emph{support} of $\beta$ to be the set $\{i \mid \b_i \ne \delta\}$. Let $I$ be the support of $\omega$, and note that $I$ is also the support of each element of $\omega^{N_\alpha}$. Now $|\omega^{N_{\alpha}}|=\prod_{i\in I}|\omega_i^{\soc(L)_\delta}|$ and each term in the product is greater than $1$ since $\omega_i\neq\delta$. But $|\omega^{N_\alpha}| = |\omega^{G_\alpha}|$ is a prime, so we must have $I=\{i\}$ for some $i$.

Since $N$ is transitive on $\Omega$ we have $G=NG_\alpha$ and so $G_\alpha$ also induces the group $K$ on the set of $k$ simple direct factors of $N$. Therefore, for each $j$ in the $K$-orbit $i^K$, there is an element of $\omega^{G_\alpha}$ whose support is $\{j\}$. Since $\omega^{N_\alpha}=\omega^{G_\alpha}$ it follows that $i^K=i$ and thus $k=2$ and $K=1$, so 
$\soc(L) \times \soc(L) \normeq G \leqs L \times L$.

Without loss of generality, we may assume that $\omega=(\gamma,\delta)$ where $|\gamma^{\soc(L)_{\delta}}|$ is a prime. If $g = (g_1,g_2)\in G$ interchanges $\alpha$ and $\omega$, then $\delta^{g_2}=\delta$. Therefore, $\delta^{h_2}=\delta$ for each $(h_1,h_2)\in \langle G_\alpha,g\rangle$ and thus $G\neq \langle G_\alpha,g\rangle$. In particular, Lemma \ref{l:con} implies that the orbital graph corresponding to $\omega^{G_{\a}}$ is disconnected.
\end{proof}

\begin{lemma}\label{l:qgraph}
Let $G \leqs {\rm Sym}(\Omega)$ be a finite $2'$-elusive quasiprimitive  permutation group with a non-simple socle. 
Then $G$ does not have a prime subdegree that corresponds to a connected orbital graph.
\end{lemma}

\begin{proof}
Write $N=\soc(G) = \soc(L)^k$, where $k \geqs 2$ and $L \leqs {\rm Sym}(\Delta)$ is the transitive almost simple group described in Theorems \ref{t:prim2dash} and \ref{t:qprim2dash}. Let $\alpha\in\Omega$. Seeking a contradiction, suppose that $\omega^{G_{\a}}$ is a self-paired suborbit of prime length and the corresponding orbital graph is connected. 

If $G$ is primitive then Theorem \ref{t:prim2dash} implies that $G$ induces a transitive permutation group on the set of simple direct factors of its socle, so $G$ does not have a prime subdegree by Lemma \ref{lem:prodactionsubdegs}. For the remainder, we may assume that $G$ is imprimitive, in which case the structure of $G$ is described in Theorem \ref{t:qprim2dash}. In particular, $\soc(L)=\PSL_2(p)$ with $p$ a Mersenne prime.

Now $G$ acts faithfully on a nontrivial system of imprimitivity $\mathcal{P}$ for $\Omega$, which we may identify with the Cartesian product $\Delta^k$. Let $A\in\mathcal{P}$ be the block containing $\alpha$. Without loss of generality, we may assume that $A=(\delta,\ldots,\delta)$ and $\soc(L)_\delta=C_p{:}C_{(p-1)/2}$. In particular, $|\Delta| = |\soc(L):\soc(L)_{\delta}|=p+1$. Since we are assuming that the orbital graph corresponding to $\omega^{G_{\a}}$ is connected, Lemma \ref{lem:imprimsubdeg} implies that $\omega\notin A$. Moreover, if $B = (b_1, \ldots, b_k) \in\mathcal{P}$ is the block containing $\omega$, then the same lemma also implies that $|B^{G_{\alpha}}|$ divides $|\omega^{G_{\alpha}}|$. 

As in the statement of Theorem \ref{t:qprim2dash}, 
$$(C_p{:}C_r)^k\leqslant N_\alpha < (C_p{:}C_{(p-1)/2})^k=N_A,$$
where $r$ is the product of the distinct prime divisors of $(p-1)/2$. Note that ${\rm O}_p(N_\alpha)\normeq G_\alpha$ and thus $|B^{{\rm O}_p(N_\alpha)}|$ divides $|B^{G_{\alpha}}|$. Now  ${\rm O}_p(\soc(L)_\delta)= C_p$ has one fixed point and one orbit of length $p$ on $\Delta$, hence $|B^{{\rm O}_p(N_\alpha)}|=p^\ell$, where $\ell = |\{i \mid b_i \ne \delta\}|$. Since $|\omega^{G_{\a}}|$ is a prime, we deduce that $\ell=1$ and 
$$|\omega^{G_{\alpha}}|=|B^{G_{\alpha}}| \textrm{ and } B^{G_{\alpha}}=B^{{\rm O}_p(N_\alpha)}.$$ 
In particular, each element of $B^{{\rm O}_p(N_\alpha)}$ differs from $A$ in precisely the same coordinate. Now $G$ transitively permutes the $k$ simple direct factors of $N$ and we have $G=NG_\alpha$ (since $N$ is transitive on $\Omega$), hence $G_\alpha$ also acts transitively on the factors of $N$. Therefore, for each $i \in \{1, \ldots, k\}$ there is an element of $B^{G_\alpha}$ that differs from $A$ in the $i$-th coordinate. This is a contradiction.
\end{proof}

Finally we turn our attention to $2'$-elusive biquasiprimitive groups. We define $\Delta_1$, $\Delta_2$ and $G^+$ as in the first paragraph of Section \ref{s:bquasi}.

\begin{lemma}\label{lem:sdc} 
Let $G \leqs {\rm Sym}(\Omega)$ be a finite biquasiprimitive permutation group such that $|\Omega| >2$ and the actions of $G^+$ on $\Delta_1$ and $\Delta_2$ are permutation isomorphic. If $\Gamma$ is a connected orbital graph of $G$, then $\Gamma$ is the standard double cover of a connected orbital graph of $G^+$ on $\Delta_1$.
\end{lemma}

\begin{proof}
Let $\Gamma$ be the connected orbital graph corresponding to a suborbit $\omega^{G_\alpha}$ with $\a \in \Delta_1$. Since the actions of $G^+$ on $\Delta_1$ and $\Delta_2$ are 
permutation isomorphic, there exists a bijection $\varphi:\Delta_1\rightarrow \Delta_2$ such that $\varphi(\alpha^g)= \varphi(\alpha)^g$ for all $\alpha\in\Delta_1$ and $g\in G^+$. In particular, $G_\alpha=G_{\varphi(\alpha)}$.  Let $\psi:V\Gamma \rightarrow \Delta_1\times \{0,1\}$ be the bijection such that $\psi(\beta)=(\beta,0)$ for each $\beta\in \Delta_1$ and $\psi(\beta)=(\varphi^{-1}(\beta),1)$ for each $\beta\in \Delta_2$.

By Lemma \ref{lem:imprimsubdeg}, $\omega^{G_\alpha}\subseteq \Delta_2$.  Moreover, since $|\Omega|>2$ and $\Gamma$ is connected, it follows that $\Gamma$ is $k$-regular, where $k = |\omega^{G_{\a}}| \geqs 2$. In particular, $\varphi(\alpha)\notin \omega^{G_\alpha}$. Let $O_1=\varphi^{-1}(\omega^{G_\alpha}) = \varphi^{-1}(\omega)^{G_\alpha}$, so $O_1$ is an orbit of $G_\alpha = (G^{+})_{\a}$ on $\Delta_1\setminus\{\alpha\}$. We claim that $O_1$ is self-paired (as a suborbit of $G^{+}$). To see this, first observe that $\omega^{G_\alpha} = \omega^{G_{\varphi(\alpha)}}$ is self-paired, so there exists an element $g \in G$ such that $\omega^g = \varphi(\a)$ and $\varphi(\a)^g = \omega$. Therefore, $g$ interchanges $\varphi^{-1}(\omega)$ and $\a$, and we note that $g \in G^{+}$ since $\varphi(\a), \omega \in \Delta_2$. This justifies the claim.

Recall that $G$ acts arc-transitively on $\Gamma$, so $G^+$ acts transitively on the set of arcs of $\Gamma$ of the form $(u,v)$ with $u \in \Delta_1$. Since each edge of $\Gamma$ corresponds to a unique such arc, it follows that $G^+$ is transitive on the set of edges of $\Gamma$. In particular, every edge of $\Gamma$ is of the form $\{\a^g,\omega^g\}$ for some $g \in G^+$.

Set $\delta=\varphi^{-1}(\omega)\in O_1$ and let $\Sigma$ be the orbital graph of $G^+$ on $\Delta_1$ corresponding to $O_1$. Then the edges of the standard double cover of $\Sigma$ are of the form $\{(\alpha^g,0),(\delta^g,1)\}$ for $g\in G^+$. Since $\delta^g=\varphi^{-1}(\omega)^g=\varphi^{-1}(\omega^g)$ it follows that $\{(\alpha^g,0),(\delta^g,1)\}$ is the image of the edge $\{\alpha^g,\omega^g\}$ of $\Gamma$ under $\psi$.
Therefore, $\Gamma$ is isomorphic to the standard double cover of $\Sigma$. Note that since $\Gamma$ is connected then so is $\Sigma$.
\end{proof}

\begin{lemma}\label{l:bqgraph}
Let $G \leqs {\rm Sym}(\Omega)$ be a finite $2'$-elusive biquasiprimitive permutation group with point stabiliser $H$. Then $G$ has a connected orbital graph $\Gamma$ of prime valency $p$ if and only if $(G,H) = (\PGL_2(p), C_p{:}C_s)$, where $p$ is a Mersenne prime, $C_r \leqs C_s < C_{(p-1)/2}$ and $r$ is the product of the distinct prime divisors of $(p-1)/2$. Moreover, $\Gamma$ is the standard double cover of a connected $p$-regular orbital graph of $\PSL_2(p)$.
\end{lemma}

\begin{proof}
Let $\omega^{G_{\a}}$ be a self-paired suborbit of prime length such that  the corresponding orbital graph $\Gamma$ is connected. Without loss of generality, we may assume that $\a \in \Delta_1$. Then $\omega\in\Delta_2$ by Lemma \ref{lem:imprimsubdeg}. 

First assume that $G^+$ is not faithful on at least one of its orbits, in which case the structure of $G$ is described in Theorem \ref{t:bqprim2dash}(c). In every case, we observe that the kernel of the action of $G^+$ on $\Delta_1$ is transitive on $\Delta_2$ (see the proof of Lemma \ref{lem:NG+unfaith}, for example). Therefore, any connected orbital graph arising from such a group is complete and bipartite. But $|\Delta_1|$ is not a prime, so this situation does not arise. 

For the remainder we may assume that $G^+$ is faithful on both orbits, in which case the structure of $G$ is given in Theorem \ref{t:bqprim2dash}(b). Let $N=T^k$ be the unique minimal normal subgroup of $G$, where $k \geqs 1$.  Let $K^+ \leqs S_k$ be the group induced by $G^+$ on the set of $k$ simple direct factors of $N$. Recall that $G^+=NG_{\a}$ (see Lemma \ref{lem:inGplus}). 

If $k=1$ then $G$ is almost simple and thus $(G,G_{\alpha}) = (\PGL_2(p), C_p{:}C_s)$ by Lemma \ref{lem:smallsubdegs}. Here $G_{\a}<T=G^+$ and $|N_G(G_{\a}):N_{G^+}(G_{\a})|=2$, so there exists an element $g \in N_{G}(G_{\a}) \setminus G^{+}$. Therefore, $(G^{+})_{\a} = G_{\a}  = G_{\a^g} = (G^{+})_{\a^g}$ and $\a^g \in \Delta_2$, so  
the actions of $G^+$ on $\Delta_1$ and $\Delta_2$ are permutation isomorphic. Therefore, Lemma \ref{lem:sdc} implies that $\Gamma$ is the standard double cover of a connected $p$-regular orbital graph of $T$ on $\Delta_1$.

To complete the proof, we may assume that $k \geqs 2$. Once again, the actions of $G^+$ on $\Delta_1$ and $\Delta_2$ are permutation isomorphic, so the lengths of the orbits of $G_\alpha$ on $\Delta_2$ are the same as those on $\Delta_1$. By Lemma \ref{lem:sdc}, $\Gamma$ is the standard double cover of a connected orbital graph $\Sigma$ of $G^+$ on $\Delta_1$ of prime valency. 
If $T={\rm M}_{11}$ or ${}^2F_4(2)'$ then $G^+$ acts on $\Delta_1$ with its standard product action, 
so by appealing to Lemma \ref{lem:prodactionsubdegs} we deduce that $G$ does not have an appropriate self-paired suborbit. This is a contradiction.  

Finally, let us assume that $k \geqs 2$ and $T=\PSL_2(p)$. To eliminate this case, we will show that $G^{+}$ does not have a connected orbital graph on $\Delta_1$ of prime valency. Seeking a contradiction, suppose $\beta^{(G^+)_{\a}}  = \beta^{G_{\a}} \subseteq \Delta_1$ is a self-paired suborbit of prime length with connected orbital graph $\Sigma$. Here $\Delta_1$ admits a $G^+$-invariant partition $\mathcal{P}$ that can be identified with an appropriate Cartesian product $\Delta^k$. Let $A\in\mathcal{P}$  be the block containing $\a$. Without loss of generality, we may assume that $A=(\delta,\ldots,\delta)$ for some $\delta\in\Delta$. Similarly, let $B\in \mathcal{P}$ be the block containing $\beta$. 

By Lemma \ref{lem:imprimsubdeg} it follows that $A\neq B$ and $|B^{G_\alpha}|$ divides $|\beta^{G_\alpha}|$. If $\beta^{G_\alpha}\subseteq B$ then the connectivity of $\Sigma$ implies that $\mathcal{P}=\{A,B\}$, which is a contradiction since $|\Delta^k|> 2$.  Therefore $|B^{G_\alpha}|=|\beta^{G_\alpha}|$ and by applying Lemma \ref{lem:prodactionsubdegs} we deduce that $k=2$ and $K^+=1$. In particular, $G^+ \leqs {\rm PGL}_{2}(p) \times {\rm PGL}_{2}(p)$. By arguing as in the proof of Lemma \ref{lem:prodactionsubdegs} we see that $B$ differs from $A$ in precisely one coordinate. Without loss of generality we may assume that $B=(\gamma,\delta)$ for some $\gamma \ne \delta$. If $g=(g_1,g_2)\in G^+$ interchanges $\alpha$ and $\beta$, then $g$ also interchanges $A$ and $B$ and we have $\delta^{g_2}=\delta$. It follows that 
$\delta^{h_2}=\delta$ for each $(h_1,h_2)\in\langle G_\alpha,g\rangle \leqslant \langle G_A,g\rangle$. Therefore, $\langle G_\alpha,g\rangle\neq G$ and thus Lemma \ref{l:con} implies that $\Sigma$ is disconnected, a contradiction.
\end{proof}

Finally, we are now ready to prove Theorem \ref{t:graph} and Corollaries \ref{c:prime} and \ref{c:6}.

\begin{proof}[Proof of Theorem \ref{t:graph}]
Let $\Gamma$ be a finite connected graph of prime valency $p$ and let $G\leqslant \Aut(\Gamma)$ be an arc-transitive group of automorphisms so that the action of $G$ on the vertex set $V\Gamma$ is either quasiprimitive or biquasiprimitive. We may assume that $G$ is $2'$-elusive on $V\Gamma$ (otherwise case (i) or (ii) in Theorem \ref{t:graph} holds). Then $G$ is almost 
simple by Lemmas \ref{l:qgraph} and \ref{l:bqgraph}, and the result now follows by applying Lemmas \ref{lem:smallsubdegs} and \ref{l:bqgraph}.
\end{proof}

\begin{proof}[Proof of Corollary \ref{c:prime}]
Let $\Gamma$ be a finite connected graph of prime valency and let $G\leqslant \Aut(\Gamma)$ be an elusive arc-transitive group of automorphisms. Since the valency is prime, for each vertex $v \in V\Gamma$, the action of $G_v$ on the set of neighbours of $v$ is primitive. Let $N$ be a normal subgroup of $G$ and suppose that $N$ has at least three orbits on vertices. Then by \cite[Lemma 1.6]{P85}, $N$ is semiregular, which contradicts the fact that $G$ is elusive. Therefore, $G$ is either quasiprimitive or biquasiprimitive on vertices, and thus $G$ and $\Gamma$ are given by Theorem \ref{t:graph}(ii)--(v).  If $|V\Gamma|$ is a power of two then \cite[Proposition 3.2]{maru81} implies that $G$ contains a derangement of order two, which is a contradiction (the proof in \cite{maru81} applies to any vertex-transitive subgroup, not just the full automorphism group). In cases (iv) and (v) of Theorem \ref{t:graph}, note that $|G|$ is even and $|G_v|$ is odd, so once again we deduce that $G$ contains a derangement of order two. Therefore, the only possibility is the example in part (iii), hence $G={\rm M}_{11}$ and $\Gamma$ is the complete graph $K_{12}$.
\end{proof}

\begin{proof}[Proof of Corollary \ref{c:6}]
Let $k$ be the smallest integer such that there is a finite connected graph of valency $k$ with an elusive arc-transitive group of automorphisms. By \cite[Theorem 1.1]{GMPV}, $k \geqs 5$. As discussed in the introduction to \cite{GMPV}, the example in \cite[Theorem 3.5(3)]{GMPV} shows that $k \leqs 6$. Since Corollary \ref{c:prime} shows that $k=5$ is not possible, we conclude that $k=6$.
\end{proof}

\end{document}